\documentclass{amsart}
\usepackage{graphicx}
\usepackage{epstopdf}
\usepackage{tikz}
\usepackage{subcaption}
\usepackage{multirow}
\usetikzlibrary{matrix}
\newtheorem{thm}{Theorem}[section]

\newtheorem{lem}[thm]{Lemma}
\newtheorem{prop}[thm]{Proposition}
\theoremstyle{definition}

\newtheorem{exam}[thm]{Example}
\theoremstyle{remark}
\newtheorem{rem}[thm]{Remark}
\numberwithin{equation}{section}

\begin{document}

\title[2D Moore CA with new boundary conditions and its reversibility]{2D Moore CA with new boundary conditions and its reversibility}%
\author[1]{B.A. Omirov\textsuperscript{1}}
\author[2]{Sh.B. Redjepov\textsuperscript{2}}
\author[3]{J.B. Usmonov\textsuperscript{3}}
\address[1]{Institute for Advanced Study in Mathematics, Harbin Institute of Technologies, Harbin, 150001, China}%
\email{omirovb@mail.ru}%
\address[2]{Tashkent University of Information Technologies, Amir Temur street,
100200, Tashkent,  Uzbekistan}%
\email{sh.redjepov@gmail.com}%
\address[3]{National University of Uzbekistan, University street, 100174, Tashkent, Uzbekistan}%
\email{javohir0107@mail.ru}%


\begin{abstract}
 In this paper, under certain conditions we consider two-dimensional cellular automata with the Moore neighborhood. Namely, the characterization of 2D linear cellular automata defined by the Moore neighborhood with some mixed boundary conditions over the field $\mathbb{Z}_{p}$ is studied. Furthermore, we investigate the rule matrices of 2D Moore CA under some mixed boundary conditions  by applying rotation. Finally, we give the conditions under which the obtained rule matrices for  2D finite CAs are reversible.
\end{abstract}
\maketitle

\section{Introduction}\label{sec1}

It is known that a cellular automaton (CA) is a set of cells arranged in a grid of a specific shape. Each cell changes its state over time based on a predetermined set of rules determined by the states of neighboring cells. Cellular automata (CAs) have been proposed for potential applications in public-key cryptography, as well as in the fields of geography, anthropology, political science, sociology, physics, and others (refer to \cite{Adamatzky1995}, \cite{Adamatzky1999}, \cite{Adamatzky2010}).
 Cellular automata were studied in the early 1950s as a possible model for biological systems by J. Von
Neumann and Stan Ulam (\cite{Neumann1966}, \cite{Ulam1962}, \cite{Wolfram2002}).  Two most common types of CA used by different authors are: one-dimensional CA (1D CA) and two-dimensional CA (2D CA). As a famous example of 2D CA, John Conway's Game of Life (also known simply as Life) is a two-dimensional, totalistic CA that introduces more complexity than an elementary CA, since each cell in the grid has a bigger neighborhood. It is a computation-universal CA since it can effectively emulate any CA, Turing machine, or other systems that can be translated into a system known to be universal (\cite{game}, \cite{turing}). If the grid is a linear array of cells, is called {\it 1D CA} and if it is a rectangular or hexagonal grid of cells then it is called {\it 2D CA}. One-dimensional cellular automata have now been investigated in several ways \cite{Wolfram1984}.  A CA with one central cell and four near neighborhood cells is called {\it a von Neumann neighborhood CA} whereas a CA having one central cell and eight near neighborhood cells is called {\it the Moore neighborhood CA}.

A configuration of the system is an assignment of states to all the cells. Every configuration determines the next configuration via a transition rule that is local in the sense that the state of a cell at time $(t + 1)$ depends only on the states of some of its neighbors at time $t$. When the  transition rule is linear and under some boundary conditions there are several results (see \cite{Uguz2017}). Usually, 2D CA is considered with triangular, square, hexagonal, and pentagonal lattices (see \cite{Bays2009}, \cite{Siap2011}, \cite{Uguz2013}, \cite{Uguz2019}, \cite{Jumaniyozov2023}). In the paper \cite{Uguz2021}  investigated the evolution of image patterns corresponding to the uniform linear rules of 2D CA with the reflexive and adiabatic boundary conditions over $\mathbb{Z}_2$. Moreover, the linear rules of CA can be found to be some image copies of a given first image depending on the special boundary types. The mathematical representation of 2D finite cellular automata (CA) allows us to determine the description of the studied CA. The more critical aspect is determining the reversibility or irreversibility of these CAs. The reversibility is the important character of the CAs which characterizes the non existence of Gardens of Eden. A reversible cellular automaton is a cellular automaton in which every configuration has a unique predecessor. It has been demonstrated that determining the reversibility of cellular automata (CA) for dimensions greater than or equal to two is undecidable (see \cite{Kari1990},  \cite{Kari1994}, \cite{Siap2011}, \cite{Uguz2013}, \cite{Uguz2017}). This implies that, in general, obtaining the inverse of a given cellular automaton (CA) for higher dimensions through an algorithm is unattainable due to its complex structure. Consequently, it can be observed that determining inverses or cases of reversibility for 2D finite CAs is a complex challenge in the general scenario (\cite{14}).

In this paper, we study 2D linear CA for Moore neighbors on the square lattice. Since the local rule is a linear function, we obtain transition rules as matrices. Then we investigated CA under new types of boundary conditions with the p-state spin value case, i.e., over the field $\mathbb{Z}_p$. We obtain the transition rule matrices of the Moore finite CA over some mixed boundary conditions. Then we give the algorithm for computing the rank of obtained rule matrices for the Moore neighborhood. Finally, we give the conditions under which the obtained rule matrices of  2D finite CAs are reversible.

\section{Preliminary}
The 2D finite CA consists of $m\times n$ cells arranged in $m$ rows and $n$ columns, where each cell takes one of the
values of the field $\mathbb{Z}_p$. From now on, we will denote 2D finite CA order to $m\times n$ by 2D $\rm{CA}_{m\times n}$. A configuration of the system is an assignment of the states to all cells. Every configuration determines a next configuration
via a linear transition rule that is local in the sense that the state of a cell at time $(t + 1)$ depends only on the
states of some of its neighbors at the time $t$ using modulo $p$ algebra.

\begin{figure}[h]
  \centering
  \includegraphics[width=0.65\textwidth]{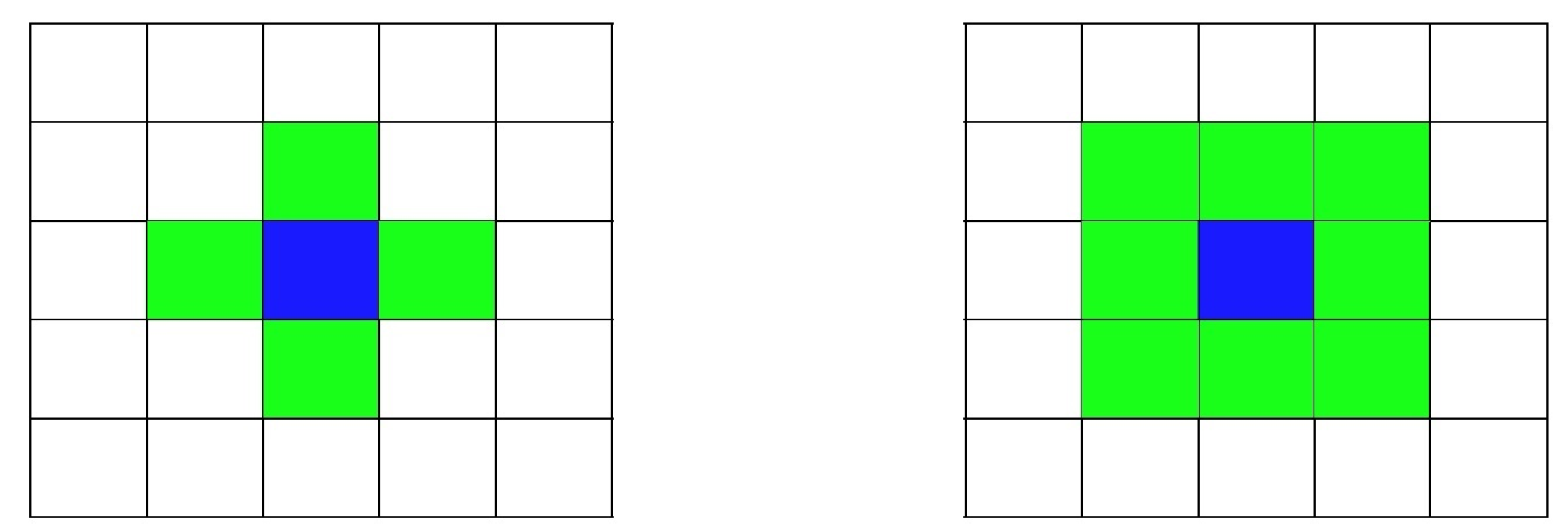}\\
  \caption{von Neumann and Moore neighborhoods}\label{neighbors}
\end{figure}

\subsection{The von Neumann and Moore neighborhood on CA lattice.}

In 2D CA's theory, there are some classic types of neighborhoods, but in this paper we only restrict ourselves to the Moore neighborhood. This neighborhood was used in the well known Conway's Game of Life. It is similar to the notion of 8-connected pixels in computer graphics. In Figure \ref{neighbors}, we illustrate the von Neumann and Moore neighborhoods. The von Neumann neighborhood the center cell is surrounded by four square cells (see Figure \ref{neighbors} (left)). The Moore neighborhood comprises eight square cells which surround the center cell $x_{(i;j)}$ (see Figure \ref{neighbors} (right)). From now on, we deal only with Moore neighborhood. Then the state $x^{(t+1)}_{i;j}$ of the cell $(i; j)$th at time $(t + 1)$ is defined by the local rule function $\Psi:\mathbb{Z}^8_p\rightarrow\mathbb{Z}_p$ as follows:

\begin{equation}{\label{1}}
\begin{array}{ccc}
  x_{i,j}^{(t+1)}&=\Psi(x_{i-1,j-1},x_{i-1,j},x_{i-1,j+1},x_{i,j+1},x_{i+1,j+1},x_{i+1,j},x_{i+1,j-1},x_{i,j-1}) &  \\
  &=ax_{i-1,j-1}^{(t)}+bx_{i-1,j}^{(t)}+cx_{i-1,j+1}^{(t)}+dx_{i,j+1}^{(t)}+ex_{i+1,j+1}^{(t)}+fx_{i+1,j}^{(t)} &  \\
   & +gx_{i+1,j-1}^{(t)}+hx_{i,j-1}^{(t)} \quad\quad\quad\quad\quad\quad\quad\quad\quad\quad\quad\quad\quad (\textrm{mod } p)
\end{array}
\end{equation}
where $a,b,c,d,e,f,g,h\in\mathbb{Z}_p^*=\mathbb{Z}_p\setminus\{0\}$

The value of each cell for the next state may not depend upon all eight neighbors.

\begin{rem}\label{rem1}
If we assume $a=c=e=g=0$, then all obtained above results hold for von Neumann neighborhood.
\end{rem}

Note that it is impossible to simulate a truly infinite lattice on a computer (unless the active region always remains finite). Therefore, we have to prescribe some boundary conditions (BC). Regarding the neighborhood of the boundary cells, four approaches exist:
\begin{itemize}
	\item If the boundary cells are connected to 0-state, then CA is called {\it null boundary (NB) CA} (see Table~\ref{tab:T1}).

	\item If the boundary cells are adjacent to each other, then CA is called {\it periodic boundary (PB) CA} (see Table~\ref{tab:T2}).

\item {\it An Adiabatic Boundary (AB) CA} is duplicating the value of the cell in an extra virtual neighbor (see Table~\ref{tab:T3}).

\item {\it A Reflexive Boundary (RB) CA} is designed for the value of the left and right neighbors to be equal concerning  the boundary cell (see Table~\ref{tab:T4}).
\end{itemize}

\begin{table}[h!]
\footnotesize
\caption{Null boundary condition on a 2D finite CA$_{3\times 3}$.}\label{tab:T1}
\begin{center}
  {\begin{tabular}{llllllllllllllllllll}\\[-2pt]

0 & \vline & & 0 && 0 && 0 & \vline  & 0\\
		\hline
		0 & \vline && $x_{(i-1,j-1)}$ && $x_{(i-1,j)}$ && $x_{(i-1,j+1)}$ & \vline & 0\\
		0 & \vline && $x_{(i,j-1)}$ && $x_{(i,j)}$ && $x_{(i,j+1)}$ & \vline & 0\\
		0 & \vline && $x_{(i+1,j-1)}$ && $x_{(i+1,j)}$ && $x_{(i+1,j+1)}$ & \vline & 0\\
		\hline
		0 & \vline && 0 && 0 && 0 & \vline & 0\\

\end{tabular}}
\end{center}
\end{table}

\begin{table}[h!]
\footnotesize
\caption{Periodic boundary condition on a 2D finite CA$_{3\times 3}$.}\label{tab:T2}
\begin{center}
{\begin{tabular}{llllllllllllllllllll}\\[-2pt]

		$x_{(i+1,j+1)}$ & \vline & & $x_{(i+1,j-1)}$ && $x_{(i+1,j)}$ && $x_{(i+1,j+1)}$ & \vline  & $x_{(i+1,j-1)}$\\
		\hline
		$x_{(i-1,j+1)}$ & \vline && $x_{(i-1,j-1)}$ && $x_{(i-1,j)}$ && $x_{(i-1,j+1)}$ & \vline & $x_{(i-1,j-1)}$\\
		$x_{(i,j+1)}$ & \vline && $x_{(i,j-1)}$ && $x_{(i,j)}$ && $x_{(i,j+1)}$ & \vline & $x_{(i.j-1)}$\\
		$x_{(i+1,j+1)}$ & \vline && $x_{(i+1,j-1)}$ && $x_{(i+1,j)}$ && $x_{(i+1,j+1)}$ & \vline & $x_{(i+1,j-1)}$\\
		\hline
		$x_{(i-1,j+1)}$ & \vline && $x_{(i-1,j-1)}$ && $x_{(i-1,j)}$ && $x_{(i-1,j+1)}$ & \vline & $x_{(i-1,j-1)}$\\
\end{tabular}}
\end{center}
\end{table}

\begin{table}[h!]
\footnotesize
\caption{ Adiabatic boundary condition on a 2D finite CA$_{3\times 3}$.}\label{tab:T3}
\begin{center}
{\begin{tabular}{llllllllllllllllllll}\\[-2pt]
	
		$x_{(i-1,j-1)}$ & \vline & & $x_{(i-1,j-1)}$ && $x_{(i-1,j)}$ && $x_{(i-1,j+1)}$ & \vline  & $x_{(i-1,j+1)}$\\
		\hline
		$x_{(i-1,j-1)}$ & \vline && $x_{(i-1,j-1)}$ && $x_{(i-1,j)}$ && $x_{(i-1,j+1)}$ & \vline & $x_{(i-1,j+1)}$\\
		$x_{(i,j-1)}$ & \vline && $x_{(i,j-1)}$ && $x_{(i,j)}$ && $x_{(i,j+1)}$ & \vline & $x_{(i,j+1)}$\\
		$x_{(i+1,j-1)}$ & \vline && $x_{(i+1,j-1)}$ && $x_{(i+1,j)}$ && $x_{(i+1,j+1)}$ & \vline & $x_{(i+1,j+1)}$\\
		\hline
		$x_{(i+1,j-1)}$ & \vline && $x_{(i+1,j-1)}$ && $x_{(i+1,j)}$ && $x_{(i+1,j+1)}$ & \vline & $x_{(i+1,j+1)}$\\
\end{tabular}}
\end{center}
\end{table}

\begin{table}[h!]
\footnotesize
\caption{  Reflexive boundary condition on a 2D finite CA$_{3\times 3}$.}\label{tab:T4}
\begin{center}
{\begin{tabular}{llllllllllllllllllll}\\[-2pt]

		$x_{(i,j)}$ & \vline & & $x_{(i,j-1)}$ && $x_{(i,j)}$ && $x_{(i,j+1)}$ & \vline  & $x_{(i,j)}$\\
		\hline
		$x_{(i-1,j)}$ & \vline && $x_{(i-1,j-1)}$ && $x_{(i-1,j)}$ && $x_{(i-1,j+1)}$ & \vline & $x_{(i-1,j)}$\\
		$x_{(i,j)}$ & \vline && $x_{(i,j-1)}$ && $x_{(i,j)}$ && $x_{(i,j+1)}$ & \vline & $x_{(i,j)}$\\
		$x_{(i+1,j)}$ & \vline && $x_{(i+1,j-1)}$ && $x_{(i+1,j)}$ && $x_{(i+1,j+1)}$ & \vline & $x_{(i+1,j)}$\\
		\hline
		$x_{(i,j)}$ & \vline && $x_{(i,j-1)}$ && $x_{(i,j)}$ && $x_{(i,j+1)}$ & \vline & $x_{(i,j)}$\\
\end{tabular}}
\end{center}
\end{table}

\section{The rule matrix of Moore CA and mixed boundary condition associated with non-bijective map}

Now, we can study the rule matrix under null boundary conditions. In order to characterize the corresponding rule, first we represent each matrix of size $m\times n$ as a column vector of size $mn\times1$. If the same rule is applied to all the cells in each evaluation, then those CA is called {\it uniform or regular}. Throughout the paper we deal with uniform CA.

Thus, the problem of finding a rule matrix of the corresponding rule is taken from the space of $m\times n$ matrices to the
space of $\mathbb{Z}_{p}^{mn}$.  In order to describe this problem more detailly we define the following map:
\[ \Phi \colon \mathbf{M}_{m\times n}(\mathbb{Z}_{p})\longrightarrow\mathbf{M}_{mn\times 1}(\mathbb{Z}_{p})\]
which takes the $t$-th state $X^{(t)}$ given by
\begin{equation}{\label{eq2}}
	C(t):=\begin{pmatrix}
		x_{11}^{(t)} & x_{12}^{(t)} & \dots & x_{1n}^{(t)}\\
		x_{21}^{(t)} & x_{22}^{(t)} & \dots & x_{2n}^{(t)}\\
		\vdots & \vdots & \vdots & \vdots\\
		x_{m1}^{(t)} & x_{m2}^{(t)} & \dots & x_{mn}^{(t)}\\
	\end{pmatrix}\longrightarrow X^{(t)}:=(x_{11}^{(t)},\dots,x_{1n}^{(t)},\dots,x_{m1}^{(t)},\dots,x_{mn}^{(t)})^{T}.
\end{equation}
where the superscript $T$ denotes the transpose and $\mathbf{M}_{m\times n}(\mathbb{Z}_{p})$ is the set of matrices with entries $\{0, 1, 2, \dots,p-1\}$.


Thus, local rules will be assumed to act on $\mathbb{Z}_{p}^{mn}$ rather than $\mathbf{M}_{m\times n}(\mathbb{Z}_{p})$.  The matrix
$C(t)$
is called \textit{the configuration matrix (or information matrix)} of the 2-D finite CA at the time $t$ and $C(0)$ is the initial information matrix of the 2-D finite CA. Therefore, one can conclude that $\Phi(C(t))=X^{(t)}$.

Using the identification (\ref{eq2}),  we can define
\[T_{\mathrm{R}}\cdot X^{(t)}=X^{(t+1)} \quad (\mathrm{mod}\ p ).\]

 This matrix $T_{\mathrm{R}}$ is called \textit{the rule matrix of 2D CA} such that  $T_{\mathrm{R}}$ operating on the current CA state $X^{(t)}$  generates the next state $X^{(t+1)}$.

Let $e_{i,j}\in \mathbf{M}_{m\times n}(\mathbb{Z}_{p}^{*})$ be the matrix units. Consider the following two sets:
$$X=\left\{e_{i,j}, \ 1\leq i\leq n, \ 1\leq j\leq m \right\},$$ $$Y=\left\{e_{0,i},e_{m+1,i},e_{j,0},e_{j,n+1}, \ 0\leq i\leq n, \ 0\leq j\leq m \right\}.
$$
We define $\alpha,\eta,\pi,\rho \colon Y\to X$ mappings by the boundary conditions in the tables \ref{tab:T1}, \ref{tab:T2}, \ref{tab:T3} and \ref{tab:T4}. Namely, $\alpha$ is adiabatic BC, $\eta$ is null BC, $\pi$ is periodic BC and $\rho$ is reflexive BC.

	 Now consider a mapping $\varphi \colon \Gamma\to\Gamma$ where $\Gamma=\left\{\alpha,\eta,\pi,\rho\right\}$.Then we study the CA under boundary conditions that depends on the mapping $\varphi$. In other words, the boundary cells are evaluated depending upon $\varphi(x), \ x\in\Gamma$ (see Figure~\ref{CAS}).  In the paper \cite{Jumaniyozov20231} for the bijective function $\varphi$ the characterization problem of 2D finite von Neumann CA is completely solved. Let us define $\phi $, $\psi$, $\tau$, $\sigma$, $\lambda$, and $\xi$ are non-bijective maps on $\Gamma$ as
\begin{equation}\label{nrfunc0}
\phi(\alpha)=\phi(\rho)=\eta, \ \ \phi(\eta)=\phi(\pi)=\rho,
\end{equation}

\begin{equation}\label{npfunc0}
\psi(\alpha)=\psi(\rho)=\eta, \ \ \psi(\eta)=\psi(\pi)=\pi,
\end{equation}

\begin{equation}\label{nafunc0}
\tau(\alpha)=\tau(\rho)=\eta, \ \ \tau(\eta)=\tau(\pi)=\alpha,
\end{equation}

\begin{equation}\label{rafunc0}
\sigma(\alpha)=\sigma(\rho)=\rho, \ \ \sigma(\eta)=\sigma(\pi)=\alpha,
\end{equation}

\begin{equation}\label{rpfunc0}
\lambda(\alpha)=\lambda(\rho)=\rho, \ \ \lambda(\eta)=\lambda(\pi)=\pi,
\end{equation}

\begin{equation}\label{pafunc0}
\xi(\alpha)=\xi(\rho)=\pi, \ \ \xi(\eta)=\xi(\pi)=\alpha,
\end{equation}
where each map describes some mixed boundary condition for the finite 2D CA.

Note that if we consider Moore neighborhood under the condition (\ref{nrfunc0}) there is ambiguity with setting boundary condition in the cells  $x_{0,n+1}$, $x_{m+1,0}$. In order to distinguish this unclearness on the condition (\ref{nrfunc0}) we define \textit{strong left null} and \textit{strong right reflexive} boundary conditions for those cells, i.e.  there is null boundary condition in the cells  $x_{0,n+1}$ and $x_{0,n}$, also the cells $x_{m+1,0}$ and $x_{m+1,1}$ have reflexive boundary condition. Later this progress carries out for the boundary conditions (\ref{npfunc0}), (\ref{nafunc0}), (\ref{rafunc0}), (\ref{rpfunc0}) and (\ref{pafunc0}).

\begin{figure}[h]
	\centering
	\includegraphics[width=7cm]{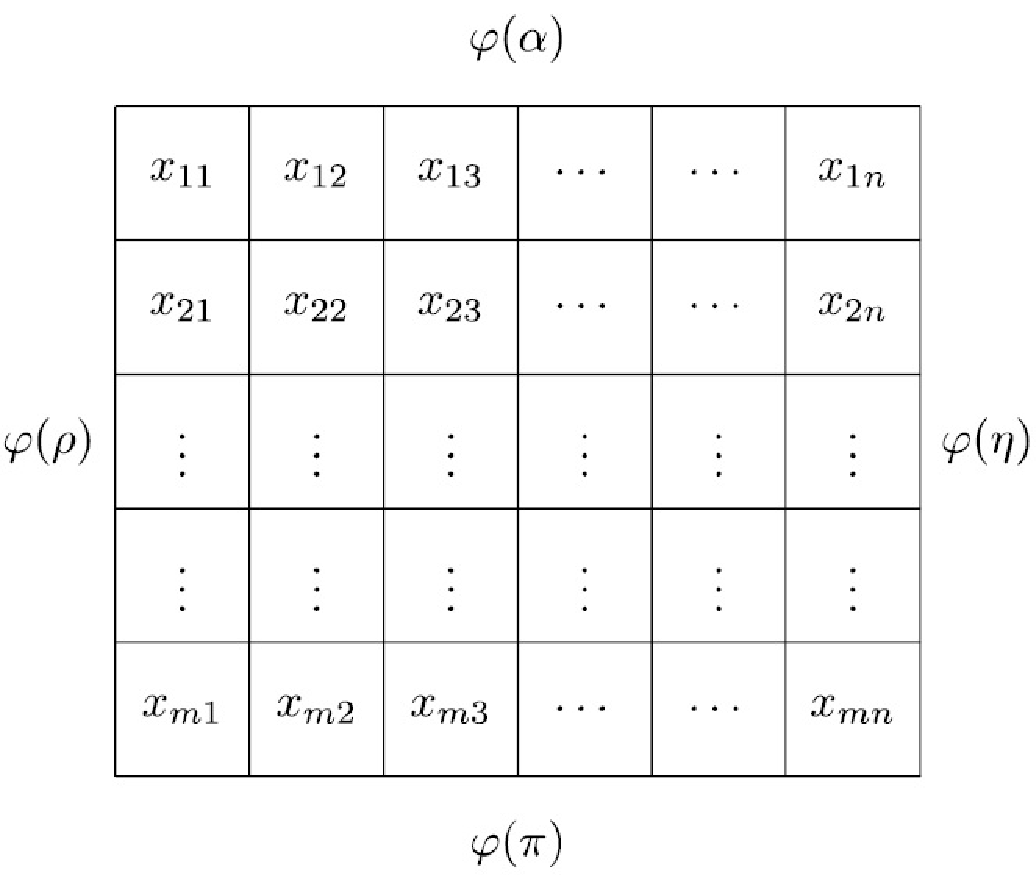}\caption{Mixing boundary condition}
	\label{CAS}
\end{figure}

To establish the transition rule matrix $T_{\text{R}}$ structure, it is needed to specify the action of	$T_{\text{R}}$ on the basis matrices $e_{i,j}$, respectively. Firstly, let us take the linear transition $T_{R}$ from $m\times n$ matrix space structure to itself. The images $T_{R}(e_{i,j})$ of $e_{i,j}$ are connected to the four nearest neighbor elements considering the Moore neighborhood. Note that the boundary condition $\phi$ does not play role for non-border cells. Hence, $T_{R}(e_{i,j})$ elements are equal to a linear sum of its eight neighbor elements. Thus, for non-border elements we have

\begin{equation}\label{eij}
  \begin{aligned}
  T_{R}(e_{i,j})&=ae_{i-1,j-1}+be_{i-1,j}+ce_{i-1,j+1}+de_{i,j+1}+ee_{i+1,j+1} \\
  &+ fe_{i+1,j}+ge_{i+1,j-1}+he_{i,j-1}.
  \end{aligned}
\end{equation}

Now, we define the action of $T_{R}$ on the border elements. All border cells have three neighbors out of the configuration, but we should define what is the boundary condition in the neighbor cells $e_{0,0},e_{0,n+1},e_{m+1,0},e_{m+1,n+1}$  of $e_{1,1},e_{1,n},e_{m,1},e_{m,n}$ out of the configuration. Without loss of generality, we obtain
$$\begin{array}{lllll}
T_{R}^{\phi}(e_{1,1})&=&d e_{1,2}+e e_{2,2}+f e_{2,1},\\
T_{R}^{\phi}(e_{1,n})&=&(h+d) e_{1,n-1}+(e+g) e_{2,n-1}+f e_{2,n},\\
T_{R}^{\phi}(e_{m,1})&=&(b+f)e_{m-1,1}+(c+e+g) e_{m-1,2}+d e_{m,2},\\
T_{R}^{\phi}(e_{m,n})&=&(a+c+e+g)e_{m-1.n-1}+(b+f)e_{m-1,n}+(d+h)e_{m,n-1},
\end{array}$$

where $a,b,c,d,e,f,g,h\in\mathbb{Z}_{p}$.

Moreover, the border elements excepting $e_{1,1},e_{1,n},e_{m,1},e_{m,n}$ have three neighbors out of the configuration. Thus, we get the following:
$$\begin{array}{lllll}
T_{R}^{\phi}(e_{1,i})&=&d e_{1,i+1}+e e_{2,i+1}+f e_{2,i}+g e_{2,i-1}+h e_{1,i-1},\\
T_{R}^{\phi}(e_{m,i})&=&(a+g)e_{m-1,i-1}+(b+f)e_{m-1,i}+(c+e)e_{m-1,i+1}\\
&&+d e_{m,i+1}+h e_{m,i-1},\\
T_{R}^{\phi}(e_{j,1})&=&b e_{j-1,1}+c e_{j-1,2}+d e_{j,2}+e e_{j+1,2}+f e_{j+1,1},\\
T_{R}^{\phi}(e_{j,n})&=&(a+c)e_{j-1,n-1}+b e_{j-1,n}+(d+h)e_{j,n-1}\\
&&+(e+g)e_{j+1,n-1}+f e_{j+1,n},
\end{array}$$
where $2\leq i\leq n-1, \ 2\leq j\leq m-1$ and $a,b,c,d,e,f,g,h\in\mathbb{Z}_{p}$.

Set
$$
P=\begin{pmatrix}
			0 & 1 & 0 & \dots & 0 & 0\\
			0 & 0 & 1 & \dots & 0 & 0\\
			0 & 0 & 0 & \dots & 0 & 0\\
			\vdots & \vdots & \vdots & \vdots & \vdots & \vdots\\
			0 & 0 & 0 & \dots & 0 & 1\\
			0 & 0 & 0 & \dots & 0 & 0\\
		\end{pmatrix}, \ \ Q=\begin{pmatrix}
			0 & 0 & 0 & \dots & 0 & 0\\
			1 & 0 & 0 & \dots & 0 & 0\\
			0 & 1 & 0 & \dots & 0 & 0\\
			\vdots & \vdots & \vdots & \vdots & \vdots & \vdots\\
			0 & 0 & 0 & \dots & 0 & 0\\
			0 & 0 & 0 & \dots & 1 & 0\\
		\end{pmatrix},$$

$$  A=dP+hQ,\quad  B=fI+eP+gQ,\quad C=bI+cP+aQ,$$
where $P, Q, I\in\mathbf{M}_{n\times n}(\mathbb{Z}_{p})$, $I$ is the identity matrix and $a,b,c,d, e, f, g, h\in\mathbb{Z}_{p}^{*}$.

Let $\phi$ be the function in (\ref{nrfunc0}), then the following result is true.

\begin{thm}\label{thm0} Let $T_{\mathrm{R}}^{\phi} \colon \mathbb{Z}_{p}^{mn} \to \mathbb{Z}_{p}^{mn}$ be  the rule matrix which takes the finite Moore CA over the configuration $C(t)$ of order $m\times n$ to the configuration $C(t+1)$ under the boundary condition of ${\phi}$. Then $T_{\mathrm{R}}^{\phi}$ has the following matrix form:
	
	\begin{equation}\label{T0}
		T_{\mathrm{R}}^{\phi}=\begin{pmatrix}
			A_1 & B_1 & O & O &\dots  &O& O & O\\
			C_1 & A_1 & B_1 &O& \dots &O & O & O\\
			O & C_1 & A_1 & B_1 &\dots &O & O & O\\
			\vdots & \vdots & \vdots & \vdots &\ddots &\vdots & \vdots & \vdots\\
			O & O & O &O&  \dots  &C_1 & A_1 & B_1\\
			O & O & O & O & \dots  &O& D_1 & A_1\\
		\end{pmatrix},
	\end{equation}
		where $  A_1=A+d\epsilon_{n,n-1}, \ \ B_1=B+e \epsilon_{n,n-1}, \ \  C_1=C+c\epsilon_{n,n-1}, \ \ D_1=B+C+(c+e)\epsilon_{n,n-1}+g\epsilon_{1,2},$\\
$O, \epsilon_{1,2}, \epsilon_{n,n-1}\in \mathbf{M}_{n\times n}(\mathbb{Z}_{p}), \, O\, \mbox{is the zero matrix and} \, \epsilon_{1,2}, \epsilon_{n,n-1} \mbox{are unit matrices}.$
\end{thm}

\begin{proof}
Firstly, let us take the linear transition $T_{R}^\varphi \colon \mathbf{M}_{m\times n}(\mathbb{Z}_{p})\to \mathbf{M}_{m\times n}(\mathbb{Z}_{p})$.  The image $T_{R}^\varphi(e_{i,j})$ of $e_{i,j}$ is connected to the four nearest neighbor elements considering the von Neumann neighborhood. Hence $T_{R}^\varphi(e_{i,j})$ elements are equal to a linear sum of its five neighbor elements.

Let us denote by $E_{(i-1)n+j}=e_{i,j}, \ 1\leq i\leq m, \ 1 \leq j \leq n$, the column vector $mn \times 1$ whose has the $((i-1)n+j)$-th (or $(i,j)$-th in matrix form) entry equals to $1$ and the others are equal to zero. Then we have
	$$T_{R}^{\phi}\cdot\left(\begin{smallmatrix}
		E_{1}\\
		\vdots\\
		E_{n}\\
		\vdots\\
		\vdots\\
		\vdots\\
		E_{mn}\\
	\end{smallmatrix}\right)=T_{R}^{\phi}\cdot\left(\begin{smallmatrix}
		e_{1,1}\\
		\vdots\\
		e_{1,n}\\
		\vdots\\
		e_{i,j}\\
		\vdots\\
		e_{m,1}\\
		\vdots\\
		e_{m,n}\\
	\end{smallmatrix}\right)=\left(\begin{smallmatrix}
		T_R^{\phi}(e_{1,1})\\
		\vdots\\
		T_R^{\phi}(e_{1,n})\\
		\vdots\\
		T_R^{\phi}(e_{i,j})\\
		\vdots\\
		T_R^{\phi}(e_{m,1})\\
		\vdots\\
		T_R^{\phi}(e_{m,n})\\
	\end{smallmatrix}\right)=\left(\begin{smallmatrix}
		d e_{1,2}+e e_{2,2}+f e_{2,1}\\
		\vdots\\
		(h+d) e_{1,n-1}+(e+g) e_{2,n-1}+f e_{2,n}\\
		\vdots\\
		ae_{i-1,j-1}+be_{i-1,j}+ce_{i-1,j+1}+de_{i,j+1}\\
+ee_{i+1,j+1}+fe_{i+1,j}+ge_{i+1,j-1}+he_{i,j-1}\\
		\vdots\\
		(b+f)e_{m-1,1}+(c+e+g) e_{m-1,2}+d e_{m,2}\\
		\vdots\\
		(a+c+e+g)e_{m-1.n-1}+(b+f)e_{m-1,n}+\\+(d+h)e_{m,n-1}\\
	\end{smallmatrix}\right)=$$

$$\begin{pmatrix}
		\begin{smallmatrix}
			0 & d & 0 & \dots & 0 & 0\\
			h & 0 & d & \dots & 0 & 0\\
			0 & h & 0 & \dots & 0 & 0\\
			\vdots & \vdots & \vdots & \vdots & \vdots & \vdots\\
			0 & 0 & 0 & \dots & 0 & d\\
			0 & 0 & 0 & \dots & h+d & 0\\
		\end{smallmatrix}
		& \begin{smallmatrix}
			f & e & 0 & \dots & 0 & 0\\
			g & f & e & \dots & 0 & 0\\
			0 & g & f & \dots & 0 & 0\\
			\vdots & \vdots & \vdots & \vdots & \vdots & \vdots\\
			0 & 0 & 0 & \dots & f & e\\
			0 & 0 & 0 & \dots & g+e & f\\
		\end{smallmatrix} &\dots &
		\begin{smallmatrix}
			0 & 0 & 0 & \dots & 0 & 0\\
			0 & 0 & 0 & \dots & 0 & 0\\
			0 & 0 & 0 & \dots & 0 & 0\\
			\vdots & \vdots & \vdots & \vdots & \vdots & \vdots\\
			0 & 0 & 0 & \dots & 0 & 0\\
			0 & 0 & 0 & \dots & 0 & 0\\
		\end{smallmatrix} & \begin{smallmatrix}
			0 & 0 & 0 & \dots & 0 & 0\\
			0 & 0 & 0 & \dots & 0 & 0\\
			0 & 0 & 0 & \dots & 0 & 0\\
			\vdots & \vdots & \vdots & \vdots & \vdots & \vdots\\
			0 & 0 & 0 & \dots & 0 & 0\\
			0 & 0 & 0 & \dots & 0 & 0\\
		\end{smallmatrix} \\
		
		&  &  &  &  & \\
		
		\begin{smallmatrix}
			b & c & 0 & \dots & 0 & 0\\
			a & b & c & \dots & 0 & 0\\
			0 & a & b & \dots & 0 & 0\\
			\vdots & \vdots & \vdots & \vdots & \vdots & \vdots\\
			0 & 0 & 0 & \dots & b & c\\
			0 & 0 & 0 & \dots & a+c & b\\
		\end{smallmatrix} &
		\begin{smallmatrix}
			0 & d & 0 & \dots & 0 & 0\\
			h & 0 & d & \dots & 0 & 0\\
			0 & h & 0 & \dots & 0 & 0\\
			\vdots & \vdots & \vdots & \vdots & \vdots & \vdots\\
			0 & 0 & 0 & \dots & 0 & d\\
			0 & 0 & 0 & \dots & h+d & 0\\
		\end{smallmatrix}  &\dots &
		\begin{smallmatrix}
			0 & 0 & 0 & \dots & 0 & 0\\
			0 & 0 & 0 & \dots & 0 & 0\\
			0 & 0 & 0 & \dots & 0 & 0\\
			\vdots & \vdots & \vdots & \vdots & \vdots & \vdots\\
			0 & 0 & 0 & \dots & 0 & 0\\
			0 & 0 & 0 & \dots & 0 & 0\\
		\end{smallmatrix} &
		\begin{smallmatrix}
			0 & 0 & 0 & \dots & 0 & 0\\
			0 & 0 & 0 & \dots & 0 & 0\\
			0 & 0 & 0 & \dots & 0 & 0\\
			\vdots & \vdots & \vdots & \vdots & \vdots & \vdots\\
			0 & 0 & 0 & \dots & 0 & 0\\
			0 & 0 & 0 & \dots & 0 & 0\\
		\end{smallmatrix}\\
		
		&  &  &  &  & \\		
		
		\vdots &\vdots & \ddots & \vdots & \vdots \\
		
		&  &  &  &  & \\

		\begin{smallmatrix}
			0 & 0 & 0 & \dots & 0 & 0\\
			0 & 0 & 0 & \dots & 0 & 0\\
			0 & 0 & 0 & \dots & 0 & 0\\
			\vdots & \vdots & \vdots & \vdots & \vdots & \vdots\\
			0 & 0 & 0 & \dots & 0 & 0\\
			0 & 0 & 0 & \dots & 0 & 0\\
		\end{smallmatrix} &
		\begin{smallmatrix}
			0 & 0 & 0 & \dots & 0 & 0\\
			0 & 0 & 0 & \dots & 0 & 0\\
			0 & 0 & 0 & \dots & 0 & 0\\
			\vdots & \vdots & \vdots & \vdots & \vdots & \vdots\\
			0 & 0 & 0 & \dots & 0 & 0\\
			0 & 0 & 0 & \dots & 0 & 0\\
		\end{smallmatrix}
		& \dots
		&
		\begin{smallmatrix}
			0 & d & 0 & \dots & 0 & 0\\
			h & 0 & d & \dots & 0 & 0\\
			0 & h & 0 & \dots & 0 & 0\\
			\vdots & \vdots & \vdots & \vdots & \vdots & \vdots\\
			0 & 0 & 0 & \dots & 0 & d\\
			0 & 0 & 0 & \dots & h+d & 0\\
		\end{smallmatrix}
		& \begin{smallmatrix}
			f & e & 0 & \dots & 0 & 0\\
			g & f & e & \dots & 0 & 0\\
			0 & g & f & \dots & 0 & 0\\
			\vdots & \vdots & \vdots & \vdots & \vdots & \vdots\\
			0 & 0 & 0 & \dots & f & e\\
			0 & 0 & 0 & \dots & g+e & f\\
		\end{smallmatrix}\\
		
		&  &  &  &  & \\
		
		\begin{smallmatrix}
			0 & 0 & 0 & \dots & 0 & 0\\
			0 & 0 & 0 & \dots & 0 & 0\\
			0 & 0 & 0 & \dots & 0 & 0\\
			\vdots & \vdots & \vdots & \vdots & \vdots & \vdots\\
			0 & 0 & 0 & \dots & 0 & 0\\
			0 & 0 & 0 & \dots & 0 & 0\\
		\end{smallmatrix} &
		\begin{smallmatrix}
			0 & 0 & 0 & \dots & 0 & 0\\
			0 & 0 & 0 & \dots & 0 & 0\\
			0 & 0 & 0 & \dots & 0 & 0\\
			\vdots & \vdots & \vdots & \vdots & \vdots & \vdots\\
			0 & 0 & 0 & \dots & 0 & 0\\
			0 & 0 & 0 & \dots & 0 & 0\\
		\end{smallmatrix}
		& \dots
		&
		\begin{smallmatrix}
			b+f & c+e+g & 0 & \dots & 0 & 0\\
			a+g & b+f & c+e & \dots & 0 & 0\\
			0 & a+g & b+f & \dots & 0 & 0\\
			\vdots & \vdots & \vdots & \vdots & \vdots & \vdots\\
			0 & 0 & 0 & \dots & b+f & c+e\\
			0 & 0 & 0 & \dots & a+g+c+e & b+f\\
		\end{smallmatrix} &
		\begin{smallmatrix}
			0 & d & 0 & \dots & 0 & 0\\
			h & 0 & d & \dots & 0 & 0\\
			0 & h & 0 & \dots & 0 & 0\\
			\vdots & \vdots & \vdots & \vdots & \vdots & \vdots\\
			0 & 0 & 0 & \dots & 0 & d\\
			0 & 0 & 0 & \dots & h+d & 0\\
		\end{smallmatrix} \\
	\end{pmatrix}\cdot\begin{pmatrix}
		\begin{smallmatrix}
        E_{1}\\
        E_{2}\\
		\vdots\\
        \\
		E_{n}\\
        \end{smallmatrix}\\
        \vspace{0.1cm}\\
        \begin{smallmatrix}
		E_{n+1}\\
        E_{n+2}\\
		\vdots\\
        \\
		E_{2n}\\
        \end{smallmatrix}\\
        \vspace{0.1cm}\\
		\vdots\\
        \\
        \begin{smallmatrix}
		E_{(m-2)n+1}\\
		E_{(m-2)n+2}\\
		\vdots\\
        \\
		E_{(m-1)n}\\
        \end{smallmatrix}\\
        \\
        \begin{smallmatrix}
		E_{(m-1)n+1}\\
		E_{(m-1)n+2}\\
		\vdots\\
        \\
		E_{mn}\\
        \end{smallmatrix}
	\end{pmatrix}$$ $$=\begin{pmatrix}
			A_1 & B_1 & O & O &\dots  &O& O & O\\
			C_1 & A_1 & B_1 &O& \dots &O & O & O\\
			O & C_1 & A_1 &B_1 &\dots &O & O & O\\
			\vdots & \vdots & \vdots & \vdots &\ddots &\vdots & \vdots & \vdots\\
			O & O & O &O&  \dots  &C_1& A_1 & B_1\\
			O & O & O & O & \dots  &O& D_1 & A_1\\
		\end{pmatrix}\cdot \begin{pmatrix}
		E_{1}\\
		\vdots\\
		\vdots\\
		\vdots\\
		E_{mn}\\
	\end{pmatrix}.$$
	
	Hence, the transition of the representation of matrix related to the equations above presented in \eqref{T0} is obtained. So, the proof is complete.
\end{proof}

Now we formulate $T_R$ structure for the functions $\psi$, $\tau$, $\sigma$, $\lambda$, and $\xi$ as follows:

\begin{itemize}
\item the case of $\psi$:\\[1mm]

$T_{R}^{\psi}(e_{1,1})=d e_{1,2}+e e_{2,2}+f e_{2,1},\\
T_{R}^{\psi}(e_{1,n})=d e_{1,1}+e e_{2,1}+f e_{2,n}+g e_{2,n-1}+h e_{1,n-1},\\
T_{R}^{\psi}(e_{m,1})=b e_{m-1,1}+c e_{m-1,2}+d e_{m,2}+e e_{1,2}+f e_{1,1}+g e_{1,n},\\
T_{R}^{\psi}(e_{m,n})=a e_{m-1,n-1}+b e_{m-1,n}+c e_{m-1,1}+d e_{m,1}+e e_{1,1}+f e_{1,n}+g e_{1,n-1}+h e_{m,n-1},\\
T_{R}^{\psi}(e_{1,i})=d e_{1,i+1}+e e_{2,i+1}+f e_{2,i}+g e_{2,i-1}+h e_{1,i-1},\\
T_{R}^{\psi}(e_{m,i})=a e_{m-1,i-1}+b e_{m-1,i}+c e_{m-1,i+1}+d e_{m,i+1}+e e_{1,i+1}+f e_{1,i}+g e_{1,i-1}+h e_{m,i-1},\\
T_{R}^{\psi}(e_{j,1})=b e_{j-1,1}+c e_{j-1,2}+d e_{j,2}+e e_{j+1,2}+f e_{j+1,1},\\
T_{R}^{\psi}(e_{j,n})=a e_{j-1,n-1}+b e_{j-1,n}+c e_{j-1,1}+d e_{j,1}+e e_{j+1,1}+f e_{j+1,n}+ g e_{j+1,n-1}+h e_{j,n-1};
$

\item  the case of $\tau$:\\[1mm]
$\begin{array}{llllll}
T_{R}^{\tau}(e_{1,1})&=&d e_{1,2}+e e_{2,2}+f e_{2,1},\\
T_{R}^{\tau}(e_{1,n})&=&d e_{1,n}+(e+f) e_{2,n}+g e_{2,n-1}+h e_{1,n-1},\\
T_{R}^{\tau}(e_{m,1})&=&b e_{m-1,1}+c e_{m-1,2}+(d+e) e_{m,2}+(f+g) e_{m,1},\\
T_{R}^{\tau}(e_{m,n})&=&a e_{m-1,n-1}+(b+c) e_{m-1,n}+(d+e+f) e_{m,n}+(g+h) e_{m,n-1},\\
T_{R}^{\tau}(e_{1,i})&=&d e_{1,i+1}+e e_{2,i+1}+f e_{2,i}+g e_{2,i-1}+h e_{1,i-1},\\
T_{R}^{\tau}(e_{m,i})&=&a e_{m-1,i-1}+b e_{m-1,i}+c e_{m-1,i+1}+(d+e) e_{m,i+1}+f e_{m,i}+(g+h) e_{m,i-1},\\
T_{R}^{\tau}(e_{j,1})&=&b e_{j-1,1}+c e_{j-1,2}+d e_{j,2}+e e_{j+1,2}+f e_{j+1,1},\\
T_{R}^{\tau}(e_{j,n})&=&a e_{j-1,n-1}+(b+c) e_{j-1,n}+d e_{j,n}+(e+f) e_{j+1,n}+g e_{j+1,n-1}+h e_{j,n-1};
\end{array}$\\[1mm]

\item the case of $\sigma$:\\[1mm]
$\begin{array}{llllll}
T_{R}^{\sigma}(e_{1,1})&=&(a+c+e+g) e_{2,2}+(b+f) e_{2,1}+(h+d) e_{1,2},\\
T_{R}^{\sigma}(e_{1,n})&=&(a+c+g) e_{2,n-1}+(b+e+f) e_{2,n}+d e_{1,n}+h e_{1,n-1},\\
T_{R}^{\sigma}(e_{m,1})&=&(a+c) e_{m-1,2}+b e_{m-1,1}+(d+e+h) e_{m,2}+(f+g) e_{m,1},\\
T_{R}^{\sigma}(e_{m,n})&=&a e_{m-1,n-1}+(b+c) e_{m-1,n}+(d+e+f) e_{m,n}+(g+h) e_{m,n-1},\\
T_{R}^{\sigma}(e_{1,i})&=&(a+g) e_{2,i-1}+(b+f) e_{2,i}+(c+e) e_{2,i+1}+d e_{1,i+1}+h e_{1,i-1},\\
T_{R}^{\sigma}(e_{m,i})&=&a e_{m-1,i-1}+b e_{m-1,i}+c e_{m-1,i+1}+(d+e) e_{m,i+1}+f e_{m,i}+(g+h) e_{m,i-1},\\
T_{R}^{\sigma}(e_{j,1})&=&(a+c) e_{j-1,2}+b e_{j-1,1}+(d+h) e_{j,2}+(e+g) e_{j+1,2}+f e_{j+1,1},\\
T_{R}^{\sigma}(e_{j,n})&=&a e_{j-1,n-1}+(b+c) e_{j-1,n}+d e_{j,n}+(e+f) e_{j+1,n}+g e_{j+1,n-1}+h e_{j,n-1};
\end{array}$

\item the case of $\lambda$:\\[1mm]
$\begin{array}{llllll}
T_{R}^{\lambda}(e_{1,1})&=&(a+c+e+g) e_{2,2}+(b+f) e_{2,1}+(h+d) e_{1,2},\\
T_{R}^{\lambda}(e_{1,n})&=&(a+c+g) e_{2,n-1}+(b+f) e_{2,n}+d e_{1,1}+e e_{1,2}+h e_{1,n-1},\\
T_{R}^{\lambda}(e_{m,1})&=&(a+c)e_{m-1,2}+b e_{m-1,1}+(d+h) e_{m,2}+ e e_{1,2}+ f e_{1,1}+ g e_{1,n},\\
T_{R}^{\lambda}(e_{m,n})&=&a e_{m-1,n-1}+b e_{m-1,n}+c e_{m-1,1}+d e_{m,1}+e e_{1,1}+f e_{1,n}+g e_{1,n-1}+h e_{m,n-1},\\
T_{R}^{\lambda}(e_{1,i})&=&(a+g) e_{2,i-1}+(b+f) e_{2,i}+(c+e) e_{2,i+1}+d e_{1,i+1}+h e_{1,i-1},\\
T_{R}^{\lambda}(e_{m,i})&=&a e_{m-1,i-1}+b e_{m-1,i}+c e_{m-1,i+1}+d e_{m,i+1}+e e_{1,i+1}+f e_{1,i}+g e_{1,i-1}+h e_{m,i-1},\\
T_{R}^{\lambda}(e_{j,1})&=&(a+c) e_{j-1,2}+b e_{j-1,1}+(d+h) e_{j,2}+(e+g) e_{j+1,2}+f e_{j+1,1},\\
T_{R}^{\lambda}(e_{j,n})&=&a e_{j-1,n-1}+b e_{j-1,n}+c e_{j-1,1}+d e_{j,1}+e e_{j+1,1}+f e_{j+1,n}+ g e_{j+1,n-1}+h e_{j,n-1};
\end{array}$\\[1mm]

\item the case of $\xi$:\\[1mm]
$\begin{array}{lllllll}
T_{R}^{\xi}(e_{1,1})&=&a e_{m,n}+b e_{m,1}+c e_{m,2}+d e_{1,2}+ e e_{2,2}+f e_{2,1}+g e_{2,n}+h e_{1,n},\\
T_{R}^{\xi}(e_{1,n})&=&a e_{m,n-1}+b e_{m,n}+c e_{m,1}+d e_{1,n}+(e+f)e_{2,n}+g e_{2,n-1}+h e_{1,n-1},\\
T_{R}^{\xi}(e_{m,1})&=&a e_{m-1,n}+b e_{m-1,1}+ c e_{m-1,2}+(d+e)e_{m,2}+(f+g)e_{m,1}+h e_{m,n},\\
T_{R}^{\xi}(e_{m,n})&=&a e_{m-1,n-1}+(b+c) e_{m-1,n}+(d+e+f) e_{m,n}+(g+h) e_{m,n-1},\\
T_{R}^{\xi}(e_{1,i})&=&a e_{m,i-1}+b e_{m,i}+c e_{m,i+1}+d e_{1,i+1}+e e_{2,i+1}+f e_{2,i}+g e_{2,i-1}+h e_{1,i-1},\\
T_{R}^{\xi}(e_{m,i})&=&a e_{m-1,i-1}+b e_{m-1,i}+c e_{m-1,i+1}+(d+e) e_{m,i+1}+f e_{m,i}+(g+h) e_{m,i-1},\\
T_{R}^{\xi}(e_{j,1})&=&a e_{j-1,n}+b e_{j-1,1}+c e_{j-1,2}+d e_{j,2}+e e_{j+1,2}+f e_{j+1,1}+g e_{j+1,n}+h e_{j,n},\\
T_{R}^{\xi}(e_{j,n})&=&a e_{j-1,n-1}+(b+c) e_{j-1,n}+d e_{j,n}+(e+f) e_{j+1,n}+g e_{j+1,n-1}+h e_{j,n-1},
\end{array}$
\end{itemize}

where $2\leq i\leq n-1, \ 2\leq j\leq m-1$ and $a,b,c,d,e,f,g,h\in\mathbb{Z}_{p}$.

Applying the above forms of matrix rule we give the following theorem without proof.
\begin{thm}\label{thm20} Let the boundary conditions generated by $\psi$, $\tau$, $\sigma$, $\lambda$ and $\xi$. Then $T_{\mathrm{R}}^{\psi}$, $T_{\mathrm{R}}^{\tau}$, $T_{\mathrm{R}}^{\sigma}$, $T_{\mathrm{R}}^{\lambda}$ and $T_{\mathrm{R}}^{\xi}$ have the following matrix forms:
	\\[2mm]

$$\begin{array}{cc}
  T_{\mathrm{R}}^{\psi}=\begin{pmatrix}
			A_{np} & B_{np} & O & O &\dots  &O& O & O\\
			C_{np} & A_{np} & B_{np} &O& \dots &O & O & O\\
			O & C_{np} & A_{np} & B_{np} &\dots &O & O & O\\
			\vdots & \vdots & \vdots & \vdots &\ddots &\vdots & \vdots & \vdots\\
			O & O & O &O&  \dots  &C_{np} & A_{np} & B_{np}\\
			D_{np} & O & O & O & \dots  &O& C_{np} & A_{np}\\
		\end{pmatrix}, & T_{\mathrm{R}}^{\tau}=\begin{pmatrix}
			A_{na} & B_{na} & O & O &\dots  &O& O & O\\
			C_{na} & A_{na} & B_{na} &O& \dots &O & O & O\\
			O & C_{na} & A_{na} & B_{na} &\dots &O & O & O\\
			\vdots & \vdots & \vdots & \vdots &\ddots &\vdots & \vdots & \vdots\\
			O & O & O &O&  \dots  &C_{na} & A_{na} & B_{na}\\
			O & O & O & O & \dots  &O& C_{na} & D_{na}\\
		\end{pmatrix}, \end{array}$$

$$\begin{array}{cc}  T_{\mathrm{R}}^{\sigma}=\begin{pmatrix}
			A_{ra} & E_{ra} & O & O &\dots  &O& O & O\\
			C_{ra} & A_{ra} & B_{ra} &O& \dots &O & O & O\\
			O & C_{ra} & A_{ra} & B_{ra} &\dots &O & O & O\\
			\vdots & \vdots & \vdots & \vdots &\ddots &\vdots & \vdots & \vdots\\
			O & O & O &O&  \dots  &C_{ra} & A_{ra} & B_{ra}\\
			D_{ra} & O & O & O & \dots  &O& C_{ra} & A_{ra}\\
		\end{pmatrix}, & T_{\mathrm{R}}^{\lambda}=\begin{pmatrix}
			F_{rp} & E_{rp} & O & O &\dots  &O& O & O\\
			C_{rp} & A_{rp} & B_{rp} &O& \dots &O & O & O\\
			O & C_{rp} & A_{rp} & B_{rp} &\dots &O & O & O\\
			\vdots & \vdots & \vdots & \vdots &\ddots &\vdots & \vdots & \vdots\\
			O & O & O &O&  \dots  &C_{rp} & A_{rp} & B_{rp}\\
			D_{rp} & O & O & O & \dots  &O& C_{rp} & A_{rp}\\
		\end{pmatrix}, \end{array}$$

  $$\begin{array}{cc}T_{\mathrm{R}}^{\xi}=\begin{pmatrix}
			A_{pa} & B_{pa} & O & O &\dots  &O& O & E_{pa}\\
			C_{pa} & A_{pa} & B_{pa} &O& \dots &O & O & O\\
			O & C_{pa} & A_{pa} & B_{pa} &\dots &O & O & O\\
			\vdots & \vdots & \vdots & \vdots &\ddots &\vdots & \vdots & \vdots\\
			O & O & O &O&  \dots  &C_{pa} & A_{pa} & B_{pa}\\
			O & O & O & O & \dots  &O& C_{pa} & D_{pa}\\
		\end{pmatrix}, &
\end{array}$$
	where\\[1mm]
$\begin{array}{lllll}
  A_{np}=A+d\epsilon_{n,1}, & B_{np}=B+e \epsilon_{n,1},\\
  C_{np}=C+c\epsilon_{n,1},  &    D_{np}=B+e\epsilon_{n,1}+g\epsilon_{1,n},\\
  A_{na}=A+d\epsilon_{n,n}, & B_{na}=B+e \epsilon_{n,n},\\
  C_{na}=C+c\epsilon_{n,n},  &    D_{na}=A+B+g\epsilon_{1,1}+(d+e)\epsilon_{n,n},\\
  A_{ra}=A+h\epsilon_{1,2}+d\epsilon_{n,n}, & B_{ra}=B+g \epsilon_{1,2}+e \epsilon_{n,n},\\
  C_{ra}=C+a\epsilon_{1,2}+c\epsilon_{n,n}, & D_{ra}=A+B+g \epsilon_{1,1}+h \epsilon_{1,2}+(d+e) \epsilon_{n,n},\\
  A_{rp}=A+h\epsilon_{1,2}+d\epsilon_{n,1}, & B_{rp}=B+g \epsilon_{1,2}+e\epsilon_{n,1},\\
  C_{rp}=C+a\epsilon_{1,2}+c\epsilon_{n,1},  &    D_{rp}=B+e\epsilon_{n,1}+g\epsilon_{1,n},\\
  F_{rp}=A+h\epsilon_{1,2}+d\epsilon_{n,1}+e\epsilon_{n,2}, & E_{rp}=B+C+(a+g)\epsilon_{1,2}+c\epsilon_{n,n-1}+e\epsilon_{n,n},\\
   A_{pa}=A+h\epsilon_{1,n}+d\epsilon_{n,n}, & B_{pa}=B+g \epsilon_{1,n}+e \epsilon_{n,n},\\
  C_{pa}=C+a\epsilon_{1,n}+c\epsilon_{n,n}, & D_{pa}=A+B+g \epsilon_{1,1}+h \epsilon_{1,n}+(d+e) \epsilon_{n,n},\\
  E_{pa}=C+a\epsilon_{1,n}+c\epsilon_{n,1}, &   E_{ra}=B+C+(a+g)\epsilon_{1,2}+c\epsilon_{n,n-1}+e\epsilon_{n,n}. \\
\end{array}$
\\ where
 $O$, $\epsilon_{1,1}$, $\epsilon_{1,2}$, $\epsilon_{1,n}$, $\epsilon_{n,1}$, $\epsilon_{n,n-1}$, $\epsilon_{n,n}\in \mathbf{M}_{n\times n}(\mathbb{Z}_{p})$, \, $O$ is the zero matrix and \, $\epsilon_{1,1}$, $\epsilon_{1,2}$, $\epsilon_{1,n}$, $\epsilon_{n,1}$, $\epsilon_{n,n-1}$, $\epsilon_{n,n}$ are unit matrices.
\end{thm}
\subsection{The rule matrices for boundary conditions by rotating $\phi$ on the lattice}

In this subsection, we consider the rule matrices the boundary conditions generated by rotating $\varphi$ (more precisely, for the case $\phi$) on the lattice. We distinguish the following $4$ cases:   $\phi\equiv\phi_{0^\circ}$, $\phi_{90^\circ}$, $\phi_{180^\circ}$, $\phi_{270^\circ}$.

\begin{figure}[h]
  \centering
  \includegraphics[width=0.6\textwidth]{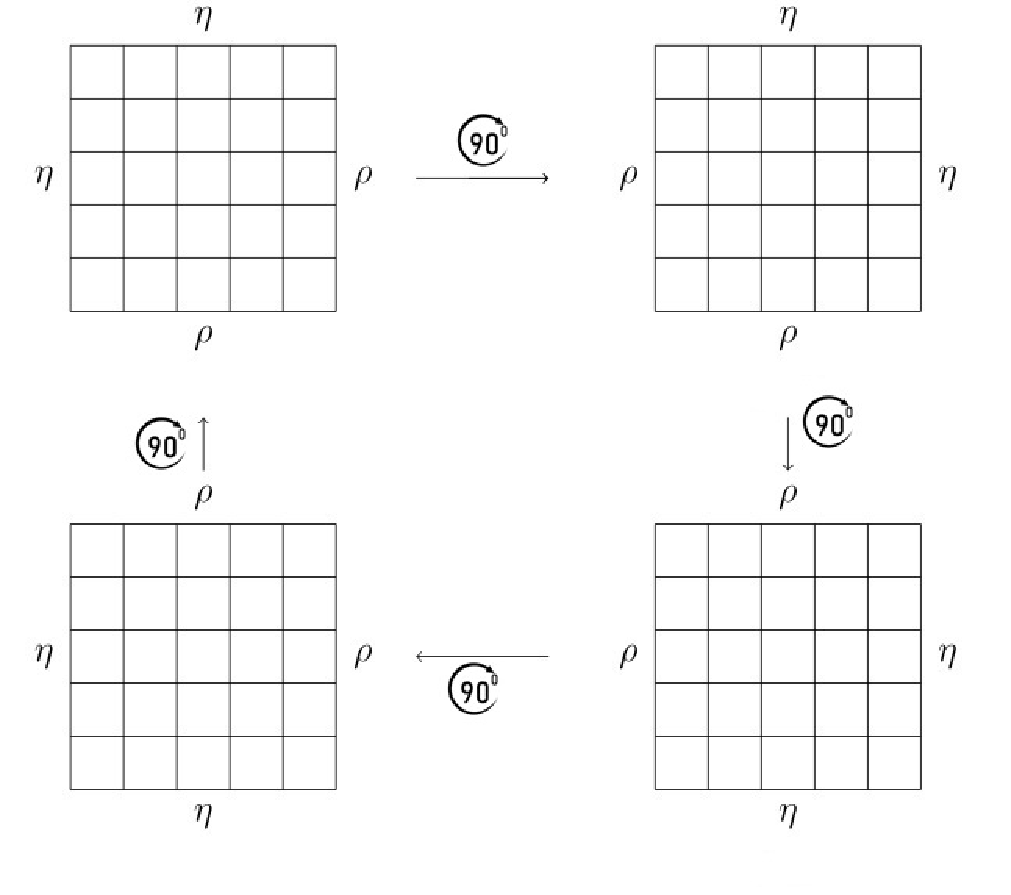}\\
  \caption{The mixed boundary conditions by rotating $90^{\circ}$ the functions $\phi$.}\label{Rotations}
\end{figure}

After rotate to $90^\circ$ degrees of the function $\phi$ on the lattice(see Figure \ref{Rotations}) then we define
$$
\phi_{90^\circ}(\alpha)=\phi_{90^\circ}(\eta)=\eta, \ \ \phi_{90^\circ}(\rho)=\phi_{90^\circ}(\pi)=\rho.
$$

Above we have defined \textit{strong left null} and \textit{strong right reflexive} for the cells $x_{0,n+1}$, $x_{m+1,0}$. If we consider ${\phi_{90^\circ}}$ then ambiguity with defining boundary conditions in the cells  $x_{0,n+1}$, $x_{m+1,0}$ moves to $x_{0,0}$, $x_{m+1,n+1}$. That is why, we define the conditions   \textit{strong up null} and \textit{strong down reflexive} in the cells  $x_{m+1,n+1}$ and $x_{0,0}$, respectively. Therefore, reflexive condition acts in the cell $x_{0,0}$ and  there is null condition in $x_{m+1,n+1}$. Then we conclude

$$\begin{array}{llll}
T_{R}^{\phi_{90^\circ}}(e_{1,1})&=&(a+g+e) e_{2,2}+(d+h) e_{1,2}+f e_{2,1},\\[1mm]
T_{R}^{\phi_{90^\circ}}(e_{1,n})&=&f e_{2,n}+g e_{2,n-1}+h e_{1,n-1},\\[1mm]
T_{R}^{\phi_{90^\circ}}(e_{m,1})&=&(a+c+e+g)e_{m-1,2}+(b+f) e_{m-1,1}+(d+h) e_{m,2},\\[1mm]
T_{R}^{\phi_{90^\circ}}(e_{m,n})&=&(a+g)e_{m-1,n-1}+(b+f)e_{m-1,n}+h e_{m,n-1},\\[1mm]
T_{R}^{\phi_{90^\circ}}(e_{1,i})&=&d e_{1,i+1}+e e_{2,i+1}+f e_{2,i}+g e_{2,i-1}+h e_{1,i-1},\\[1mm]
T_{R}^{\phi_{90^\circ}}(e_{m,i})&=&(a+g)e_{m-1,i-1}+(b+f)e_{m-1,i}+(c+e)e_{m-1,i+1}+d e_{m,i+1}\\
&&+h e_{m,i-1},\\[1mm]
\end{array}$$

$$\begin{array}{llll}
T_{R}^{\phi_{90^\circ}}(e_{j,1})&=&(a+c) e_{j-1,2}+b e_{j-1,1}+(d+h) e_{j,2}+(e+g) e_{j+1,2}+f e_{j+1,1},\\[1mm]
T_{R}^{\phi_{90^\circ}}(e_{j,n})&=&a e_{j-1,n-1}+b e_{j-1,n}+f e_{j+1,n}+g e_{j+1,n-1}+h e_{j,n-1},\\[1mm]
\end{array}$$
where $2\leq i\leq n-1, \ 2\leq j\leq m-1$ and $a,b,c,d,e,f,g,h\in\mathbb{Z}_{p}$.

For non-border elements $T_{R}(e_{i,j})$ is computed as in (\ref{eij}). The following theorem forms the transition rule matrix   $T_{\mathrm{R}}^{\phi_{90^\circ}}$.

\begin{thm}\label{thm90} Let $T_{\mathrm{R}}^{\phi_{90^\circ}} \colon \mathbb{Z}_{p}^{mn} \to \mathbb{Z}_{p}^{mn}$ be  the rule matrix which takes the finite Moore CA over the configuration $C(t)$ of order $m\times n$ to the configuration $C(t+1)$ under the boundary condition of $\phi_{90^\circ}$. Then $T_{\mathrm{R}}^{\phi_{90^\circ}}$ has the following matrix form:
	
	\begin{equation}\label{T90}
		T_{\mathrm{R}}^{\phi_{90^\circ}}=\begin{pmatrix}
			A_2 & F_2 & O & O &\dots  &O& O & O\\
			C_2 & A_2 & B_2 &O& \dots &O & O & O\\
			O & C_2 & A_2 & B_2 &\dots &O & O & O\\
			\vdots & \vdots & \vdots & \vdots &\ddots &\vdots & \vdots & \vdots\\
			O & O & O &O&  \dots  &C_2 & A_2 & B_2\\
			O & O & O & O & \dots  &O& D_2 & A_2\\
		\end{pmatrix},
	\end{equation}
	where $A_2=A+h\epsilon_{1,2}, \ \ B_2=B+g \epsilon_{1,2}, \ \    C_2=C+a\epsilon_{1,2}, \ \ D_2=B+C+(a+g)\epsilon_{1,2}, \ \ F_2=B+(a+g)\epsilon_{1,2}, $\\[1mm]
$ O, \epsilon_{1,2}\in \mathbf{M}_{n\times n}(\mathbb{Z}_{p}), \ \mbox{with} \  O \ \mbox{is the zero matrix and} \   \epsilon_{1,2} \ \mbox{is a unit matrix}.$
\end{thm}
\begin{proof}
The proof is similar to the proof of Theorem \ref{thm0}.
\end{proof}

Now we shall consider rotation to $180^\circ$ degrees of the function $\phi$ on the lattice as
$$
\phi_{180^\circ}(\eta)=\phi_{180^\circ}(\pi)=\eta, \ \ \phi_{180^\circ}(\alpha)=\phi_{180^\circ}(\rho)=\rho.
$$

If we consider rotating to $180^\circ$ degrees for the matrices $P$, $Q$, $I$, $\epsilon_{1,2}$, $\epsilon_{n,n-1}$ and the parameters $a, b, c, d, e, f, g, h \in \mathbb{Z}_p^*$ we have

$$
\begin{array}{llllllllllllll}
rot^{180^\circ}(P)&=Q,  \ \ &rot^{180^\circ}(Q)&=P, \ \  &rot^{180^\circ}(I)&=I,\\[1mm]
  rot^{180^\circ}(\epsilon_{1,2})&=\epsilon_{n,n-1},
&rot^{180^\circ}(\epsilon_{n,n-1})&=\epsilon_{1,2},
  &rot^{180^\circ}(d)&=h, \\[1mm] rot^{180^\circ}(h)&=d,
 \ \ \ & rot^{180^\circ}(e)&=a,
  &rot^{180^\circ}(g)&=c, \\[1mm] rot^{180^\circ}(f)&=b,
&rot^{180^\circ}(a)&=e,  &rot^{180^\circ}(c)&=g, \\[1mm] rot^{180^\circ}(b)&=f. & &\\[1mm]
\end{array}
$$

\begin{table}[htbp]
\footnotesize
\centering
\caption{Situation of parameters of the function $\phi$ on the boundary cells.}\label{ptable}
{\large \begin{tabular}{ccc}
  \begin{tabular}{|l|l|l|}
  \hline
  a & b & c \\
  \hline
  h &   & d \\
  \hline
  g & f & e \\
  \hline
  \end{tabular}
  & \rotatebox[origin=c]{0}{$\rightarrow$} & 
  \begin{tabular}{|l|l|l|}
  \hline
  e & f & g \\
  \hline
  d &   & h \\
  \hline
  c & b & a \\
  \hline
  \end{tabular}
\end{tabular}}
\end{table}

By above denotations we can give the following theorem.
\begin{thm}\label{thm180}
Let $T_{\mathrm{R}}^{\phi_{180^\circ}} \colon \mathbb{Z}_{p}^{mn} \to \mathbb{Z}_{p}^{mn}$ be  the rule matrix which takes the finite Moore CA over the configuration $C(t)$ of order $m\times n$ to the configuration $C(t+1)$ under the boundary condition of $\phi_{180^\circ}$. Then $T_{\mathrm{R}}^{\phi_{180^\circ}}$ has the following matrix form:

\begin{equation}\label{T180}
		T_{\mathrm{R}}^{\phi_{180^\circ}}=\begin{pmatrix}
			A_3 & D_3 & O & O &\dots  &O& O & O\\
			C_3 & A_3 & B_3 &O& \dots &O & O & O\\
			O & C_3 & A_3 & B_3 &\dots &O & O & O\\
			\vdots & \vdots & \vdots & \vdots &\ddots &\vdots & \vdots & \vdots\\
			O & O & O &O&  \dots  &C_3 & A_3 & B_3\\
			O & O & O & O & \dots  &O& C_3 & A_3\\
		\end{pmatrix},
	\end{equation}
	where $  A_3=A+h\epsilon_{1,2}, \ \ B_3=B+g \epsilon_{1,2}, \ \
  C_3=C+a\epsilon_{1,2},\ \ D_3=B+C+(a+g)\epsilon_{1,2}+c e_{n,n-1},$\\[1mm]
$O, \epsilon_{1,2}, \epsilon_{n,n-1}\in \mathbf{M}_{n\times n}(\mathbb{Z}_{p}),  O \,\mbox{is the zero matrix and} \ \ \epsilon_{1,2}, \epsilon_{n,n-1}\, \mbox{are  unit matrices}.$
\end{thm}

\begin{proof} First let us rotate the matrix with blocks then we rotate each block matrices $A_1$, $B_1$, $C_1$, $D_1$.
$$
		rot^{180^\circ}\left(T_{\mathrm{R}}^{\phi}\right)=rot^{180^\circ}\begin{pmatrix}
			A_1 & B_1 & O & O &\dots  &O& O & O\\
			C_1 & A_1 & B_1 &O& \dots &O & O & O\\
			O & C_1 & A_1 & B_1 &\dots &O & O & O\\
			\vdots & \vdots & \vdots & \vdots &\ddots &\vdots & \vdots & \vdots\\
			O & O & O &O&  \dots  &C_1 & A_1 & B_1\\
			O & O & O & O & \dots  &O& D_1 & A_1\\
		\end{pmatrix}$$
$$=\begin{pmatrix}
			rot^{180^\circ}(A_1) & rot^{180^\circ}(D_1)  &\dots  &O& O & O\\
			rot^{180^\circ}(B_1) & rot^{180^\circ}(A_1) & \dots &O & O & O\\
			O & rot^{180^\circ}(B_1) & \dots &O & O & O\\
			\vdots & \vdots & \ddots &\vdots & \vdots & \vdots\\
			O & O &   \dots  &rot^{180^\circ}(B_1) & rot^{180^\circ}(A_1) & rot^{180^\circ}(C_1)\\
			O & O &  \dots  &O& rot^{180^\circ}(B_1) & rot^{180^\circ}(A_1)\\
		\end{pmatrix}.
$$

Now we compute $rot^{180^\circ}(A_1)$, $rot^{180^\circ}(B_1)$, $rot^{180^\circ}(C_1)$ and $rot^{180^\circ}(D_1)$:\\[1mm]

$$\begin{array}{lllll}
rot^{180^\circ}(A_1)&=rot^{180^\circ}(A+d\epsilon_{n,n-1})=
rot^{180^\circ}(dP+hQ+d\epsilon_{n,n-1})\\[1mm]
&=rot^{180^\circ}(dP)+
rot^{180^\circ}(hQ)+rot^{180^\circ}(d\epsilon_{n,n-1})\\[1mm]
&=rot^{180^\circ}(d)rot^{180^\circ}(P)+rot^{180^\circ}(h)rot^{180^\circ}(Q)\\[1mm]
&+rot^{180^\circ}(d)rot^{180^\circ}(\epsilon_{n,n-1})=hQ+dP+h\epsilon_{1,2}=A+h\epsilon_{1,2}=A_3,
\end{array}$$

$$\begin{array}{lllll}
rot^{180^\circ}(B_1)&=rot^{180^\circ}(B+e\epsilon_{n,n-1})=rot^{180^\circ}(fI+eP+gQ+e\epsilon_{n,n-1})\\[1mm]
&=rot^{180^\circ}(fI)+rot^{180^\circ}(eP)+rot^{180^\circ}(gQ)+rot^{180^\circ}(e\epsilon_{n,n-1})\\[1mm]
&=rot^{180^\circ}(f)rot^{180^\circ}(I)+rot^{180^\circ}(e)rot^{180^\circ}(P)
+rot^{180^\circ}(g)rot^{180^\circ}(Q)\\[1mm]
&+rot^{180^\circ}(e)rot^{180^\circ}(\epsilon_{n,n-1})=bI+aQ+cP+a\epsilon_{1,2}=C+a\epsilon_{1,2}\\[1mm]
&=C_3,
\end{array}$$
$$\begin{array}{lllll}
rot^{180^\circ}(C_1)&=rot^{180^\circ}(C+c\epsilon_{n,n-1})=rot^{180^\circ}(bI+aQ+cP+c\epsilon_{n,n-1})\\[1mm]
&=rot^{180^\circ}(bI)+rot^{180^\circ}(aQ)
+rot^{180^\circ}(cP)+rot^{180^\circ}(c\epsilon_{n,n-1})\\[1mm]
&=rot^{180^\circ}(b)rot^{180^\circ}(I)+rot^{180^\circ}(a)rot^{180^\circ}(Q)+rot^{180^\circ}(c)rot^{180^\circ}(P)\\[1mm]
&+rot^{180^\circ}(c)rot^{180^\circ}(\epsilon_{n,n-1})\\[1mm]
&=fI+eP+gQ+g\epsilon_{1,2}=B+a\epsilon_{1,2}=B_3,
\end{array}$$

$$\begin{array}{lllll}
rot^{180^\circ}(D_1)&=rot^{180^\circ}(B+C+(c+e)e_{n,n-1}+g\epsilon_{1,2})=rot^{180^\circ}(B)
+rot^{180^\circ}(C)\\[1mm]
&+rot^{180^\circ}((c+e)e_{n,n-1})+rot^{180^\circ}(g\epsilon_{1,2})=C+B\\[1mm]
&+rot^{180^\circ}(c+e)rot^{180^\circ}(\epsilon_{n,n-1})+rot^{180^\circ}(g)rot^{180^\circ}(\epsilon_{1,2})\\[1mm]
&=C+B+(a+g)\epsilon_{1,2}+c\epsilon_{n,n-1}=D_3.
\end{array}$$

Hence, the matrix has the form
$$\begin{pmatrix}
			A_3 & D_3 & O & O &\dots  &O& O & O\\
			B_3 & A_3 & C_3 &O& \dots &O & O & O\\
			O & B_3 & A_3 & C_3 &\dots &O & O & O\\
			\vdots & \vdots & \vdots & \vdots &\ddots &\vdots & \vdots & \vdots\\
			O & O & O &O&  \dots  &B_3 & A_3 & C_3\\
			O & O & O & O & \dots  &O& B_3 & A_3\\
		\end{pmatrix}.
$$
\end{proof}

By rotating $270^{\circ}$ the function $\phi$ we obtain the following results similar those given above.
\begin{thm}\label{thm270} Let $T_{\mathrm{R}}^{\phi_{270^\circ}} \colon \mathbb{Z}_{p}^{mn} \to \mathbb{Z}_{p}^{mn}$ be  the rule matrix which takes the finite Moore CA over the configuration $C(t)$ of order $m\times n$ to the configuration $C(t+1)$ under the boundary condition of $rot^{270^\circ}(\phi(\alpha))=rot^{270^\circ}(\phi(\rho))=\eta, \ \ rot^{270^\circ}(\phi(\theta))=rot^{270^\circ}(\phi(\pi))=\rho$. Then $T_{\mathrm{R}}^{\phi_{270^\circ}}$ has the following matrix form:
	
	\begin{equation}\label{equa1}
		T_{\mathrm{R}}^{\phi_{270^\circ}}=\begin{pmatrix}
			A_4 & D_4 & O & O &\dots  &O& O & O\\
			C_4 & A_4 & B_4 &O& \dots &O & O & O\\
			O & C_4 & A_4 & B_4 &\dots &O & O & O\\
			\vdots & \vdots & \vdots & \vdots &\ddots &\vdots & \vdots & \vdots\\
			O & O & O &O&  \dots  &C_4 & A_4 & B_4\\
			O & O & O & O & \dots  &O& F_4 & A_4\\
		\end{pmatrix},
	\end{equation}
	where $ A_4=A+d\epsilon_{n,n-1}, \ \ B_4=B+e \epsilon_{n,n-1}, \ \ C_4=C+c\epsilon_{n,n-1}, \ \ D_4=B+C+(c+e)\epsilon_{n,n-1},$    $F_4=C+(c+e)\epsilon_{n,n-1}, O, \epsilon_{n,n-1}\in \mathbf{M}_{n\times n}(\mathbb{Z}_{p}),  O \, \mbox{is the zero matrix and}\\  \epsilon_{n,n-1}\, \mbox{is a unit matrix}.$
\end{thm}

The proof of the theorem is similar to the proof of Theorem \ref{thm90}.

\section{Dynamics of CAs}

The global transition function is a defining characteristic of a cellular automaton. This function shows how each configuration is changed in one time step. In our case, the global transition function is $T_R: \mathbb{Z}_p^{mn}\to \mathbb{Z}_p^{mn}$, i.e. the rule matrix of CA, where $\mathbb{Z}_p^{mn}$ are all configurations in 2D CA$_{m\times n}$. Consequently, a cellular automaton (CA) can be conceptualized as a discrete-time dynamical system denoted by $\langle T_R, Z_p\rangle$.  The attractors explain how a dynamical system behaves assymptotically. Attractors are states of the system towards which the system is attracted. The system may converge to a specific fixed-point attractor or to a periodic limit cycle attractor. Sometimes, the system represents chaotic behavior. Limit sets have been proposed as potential formalizations of attractors within the cellular automata theory. In the context of a cellular automaton, a limit set encompasses all configurations that can arise after arbitrarily long computations. The concept of "nilpotency" plays a crucial role in the dynamical behavior of CAs. A cellular automaton is called nilpotent if its limit set contains just one configuration.  \cite{Culik(1990)} have shown that for two or more dimension, the nilpotency of CAs is undecidable. In this section, we study some aspects of dynamical systems of 2D CA with mixed boundary conditions. In particular, for reasons of convenience, we consider the von Neumann neighborhood, and, in accordance with Remark \ref{rem1}, we assume $a=c=e=g=0$.

 If a square matrix $T$ is nilpotent, it means that there exists a positive integer $k$ such that $T^k=O$, where
$O$ is the zero matrix. In fact, for a given rule matrix  $T$ of 2D CA$_{m\times n}$ we have the following: if $T$ is nilpotent, then for any initial configuration, the trajectory will ultimately converge to the zero state.

\begin{prop}  Let If $T_{\mathrm{R}}^{\phi}$ be a rule of 2D CA with the boundary condition (\ref{nrfunc0}) and $d=f=0$. Then $T_{\mathrm{R}}^{\phi}$ is nilpotent.
\end{prop}
The proof of this result leads from the fact that $T_{\mathrm{R}}^{\phi}$ is a strictly lower triangular matrix under the condition $d=f=0$.

On the other hand, the dynamical system $\langle T_R, Z_p\rangle$ is over the finite field $\mathbb{Z}_p$ and the field with the usual topology induced by the metric is not compact. That is why, the dynamics of 2D CA over $\mathbb{Z}_p$ is different from real or complex cases. Now, let us find fixed points. Since $T_{\mathrm{R}}^{\phi}$ is a linear operator, $0\in \mathbb{Z}_p^{mn}$ is a fixed point. Therefore, we check for other fixed points and we shall analyse the homogeneous linear system
\begin{equation}\label{system}
(T_{\mathrm{R}}^{\phi}-I')x=0,
\end{equation}

where $I'\in M_{mn\times mn}(\mathbb{Z}_p)$ is an identity matrix.

The block matrix form of $T_{\mathrm{R}}^{\phi}-I'$ is
\begin{equation}\label{matrix}T_{\mathrm{R}}^{\phi}-I'=\begin{pmatrix}
			A_1-I & fI & O & O &\dots  &O& O & O\\
			bI & A_1-I & fI &O& \dots &O & O & O\\
			O & bI & A_1-I & fI &\dots &O & O & O\\
			\vdots & \vdots & \vdots & \vdots &\ddots &\vdots & \vdots & \vdots\\
			O & O & O &O&  \dots  &bI & A_1-I & fI\\
			O & O & O & O & \dots  &O& (b+f)I & A_1-I\\
		\end{pmatrix}.
\end{equation}

Let us assume $f=0$. Then $\det(T_{\mathrm{R}}^{\phi}-I')=(\det(A_1-I))^m$.

$$A_1-I=\begin{pmatrix}
			-1 & d & 0 & 0 &\dots  &0& 0 & 0\\
			h & -1 & d &O& \dots &0 & 0 & 0\\
			0 & h & -1 & d &\dots &0 & 0 & 0\\
			\vdots & \vdots & \vdots & \vdots &\ddots &\vdots & \vdots & \vdots\\
			0 & 0 & 0 &0&  \dots  &h & -1 & d\\
			0 & 0 & 0 & 0 & \dots  &0& h+d & -1\\
		\end{pmatrix}.$$

By using methods of evaluating higher-order determinants, for the determinant of $A_1-I$ we obtain the following result.
\begin{prop}\label{detprop} Let $A_1-I$ be a matrix $n\times n$. Then
$$\det(A_1-I)=-\Delta_{n-1}-d(h+d)\Delta_{n-2}$$

where
$\Delta_{n}=\frac{\sum\limits_{k=0}^{[\frac{n}{2}]}C_{n+1}^{2k+1}(-1)^{n-2k}(1-4hd)^{k}}{2^{n}}$ and $[\frac{n}{2}]$ is the integer part of $\frac{n}{2}$. Moreover, we can simplify this determinant under some conditions:

\begin{itemize}

\item if $d=0$, then $\det(A_1-I)=(-1)^n$;

\item if $d\neq0$ and $h=0$, then $\det(A_1-I)=(-1)^{n-1}(d^2-1)$;

\item if $dh\neq0$, $d+h=-1$, and $d\neq h$ then $\det(A_1-I)=\frac{d^{n}-h^{n}}{d-h}h$;

\item if $dh\neq0$, $d+h=-1$ and $d=h$, then $\det(A_1-I)=nd^n$.
\end{itemize}

\end{prop}

By using proposition \ref{detprop} we get solutions of the system (\ref{system}), i.e. $\rm{Fix}(T_{\mathrm{R}}^{\phi})$ is the set of fixed points $T_{\mathrm{R}}^{\phi}$ .
\begin{prop} Let $T_{\mathrm{R}}^{\phi}$ be a rule of 2D CA$_{m\times n}$ with the boundary condition (\ref{nrfunc0}) and $f=0$. Then
\begin{itemize}

\item if $d=0$, then $\rm{Fix}(T_{\mathrm{R}}^{\phi})=\{0\}$;

\item if $d\neq0,1$ and $h=0$, then $\rm{Fix}(T_{\mathrm{R}}^{\phi})=\{0\}$;

\item if $dh\neq0$, $d+h=-1$ and $d\neq h$, then $\rm{Fix}(T_{\mathrm{R}}^{\phi})=\{0\}$;

\item if $dh\neq0$, $d+h=-1$ and $d=h$, then $\rm{Fix}(T_{\mathrm{R}}^{\phi})=\{0\}$.
\end{itemize}

\end{prop}

\begin{exam} Let $m=n=3$ and $d=1$, $b=f=h=0$. Then
$$A_1-I=\begin{pmatrix}
-1&1&0&0&0&0&0&0&0\\
0&-1&1&0&0&0&0&0&0\\
0&1&-1&0&0&0&0&0&0\\
0&0&0&-1&1&0&0&0&0\\
0&0&0&0&-1&1&0&0&0\\
0&0&0&0&0&0&-1&1&0\\
0&0&0&0&0&0&0&-1&1\\
0&0&0&0&0&0&0&1&-1
\end{pmatrix}$$

According to the Proposition \ref{detprop}, $\det(A_1-I)=0$. By the theorem on homogeneous linear systems, there are some solutions of $(T_{\mathrm{R}}^{\phi}-I')x=0$ except for $0$ and these solutions are $(x_1,x_1,x_1,x_4,x_4,x_4,x_7,x_7,x_7,)$, $x_i\in \mathbb{Z}_p$.

\end{exam}

\section{Reversibility of the rule matrix $T^{\phi}_{\mathrm{R}}$.}

In the present section, we shall establish an algorithm to decide whether the 2D (linear) CA determined by the Moore rule under the boundary conditions of non-bijective function $\phi$ is reversible or not. We already have found the rule matrix $T^{\phi}_{\mathrm{R}}$ corresponding to the 2D finite CA with the function $\phi$. Thus, we can state the following relation between the column vectors $X(t)$  and the rule matrix $T^{\phi}_{\mathrm{R}}$:
\[X^{(t+1)}=T^{\phi}_{\mathrm{R}}\cdot X^{(t)} \quad (\mathrm{mod}\ p).\]

If the rule matrix $T_{\mathrm{R}}$ is non-singular, then we have
\[X^{(t)}=(T^{\phi}_{\mathrm{R}})^{-1}\cdot X^{(t+1)} \quad (\mathrm{mod}\ p).\]

Thus, our main aim is to study whether the rule matrix $T_{\mathrm{R}}^{\phi}$ in (\ref{T0}) is invertible or not. It is well known that the 2D finite CA is reversible if and only if its rule matrix $T_{\mathrm{R}}^{\phi}$ is non-singular. If the rule matrix $T_{\mathrm{R}}^{\phi}$ has full rank, then it is invertible, so the 2D finite CA is reversible, otherwise it is irreversible.

Further, we use the following auxiliary lemma.

\begin{lem}{\label{lem1}}
Let $T\in\mathbf{M}_{mn\times mn}(\mathbb{Z}_{p})$ be a matrix of the following form:
\begin{align}\label{r1}
\begin{pmatrix}
A_{m} & X & O & \dots & O & O \\
B_{m-1} & A_{m-1} & X & \dots & O & O \\
O & B_{m-2} & A_{m-2} & \dots & O & O \\
\vdots & \vdots & \vdots & \ddots & \vdots & \vdots \\
O & O & O & \dots & A_{2} & X \\
O & O & O & \dots & B_{1} & A_{1}
\end{pmatrix},
\end{align}
where all submatrices are $n\times n$, $O$ is the zero matrix. If the submatrix $X$ has full rank, then
\begin{align*}
rank(T)=(m-1)n+rank(P_{m}),
\end{align*}
with
$P_{1}=A_{1}, \ \ P_{2}=-B_{1}-A_{1}X^{-1}A_{2}, \ \ P_{k}=-P_{k-2}X^{-1}B_{k-1}-P_{k-1}X^{-1}A_{k}, \ \  k\in\{3,\ldots, m\}.$
\end{lem}

\begin{proof}
Assume that we are given the matrix $T$. Firstly, we multiply the ($m-1$)th row block by $-A_{1}X^{-1}$ from the left and add it to the last row block. Thus, we get
$$\begin{pmatrix}
A_{m} & X & O & \dots & O&O & O \\
B_{m-1} & A_{m-1} & X & \dots &O& O & O \\
\vdots & \vdots & \vdots & \ddots & \vdots & \vdots& \vdots \\
O & O & O & \dots &B_{2}& A_{2} & X \\
O & O & O & \dots & O&B_{1} & A_{1}
\end{pmatrix}\rightarrow\begin{pmatrix}
A_{m}& X & O & \dots & O & O&O \\
B_{m-1} & A_{m-1} & B_{2} & \dots & O & O&O \\
\vdots & \vdots & \vdots & \ddots & \vdots & \vdots& \vdots \\
O & O & O & \dots & B_{2}&A_{2} & X \\
O & O & O & \dots & -A_{1}X^{-1}B_{2}&B_{1}-A_{1}X^{-1}A_{2} & O
\end{pmatrix}.$$

Now, let us use the denotations $P_{2}=B_{1}-A_{1}X^{-1}A_{2}, \ Q_{2}=-A_{1}X^{-1}B_{2}$. Then, we multiply the ($m-2$)th row block by $-P_{2}X^{-1}$ from the left and add it to the last row block. Thus, we get
$$\begin{pmatrix}
A_{m} & X & O & \dots  & O& O&O \\
B_{m-1} & A_{m-1} & X & \dots  & O& O&O \\
\vdots & \vdots & \vdots & \ddots  & \vdots& \vdots& \vdots \\
O & O & O & \dots &B_{2}& A_{2} & X \\
O & O & O & \dots &Q_{2}& P_{2} & O
\end{pmatrix} \rightarrow \begin{pmatrix}
A_{m} & X & O & \dots & O& O &O\\
B_{m-1} & A_{m-1} & X & \dots & O & O&O \\
\vdots & \vdots & \vdots & \ddots  & \vdots& \vdots& \vdots \\
O & O & O & \dots &B_{2} & A_{2}&X \\
O & O & O & \dots &P_{3} & O&O
\end{pmatrix}.$$
Similarly, in the $k$-th step we multiply the $(m-k)$th row block by $-P_{k}X^{-1}$ from the left and add it to the last row block. Thus, we get
$$\begin{pmatrix}
A_{m} & X & \dots & O & O & \dots & O & O \\
B_{m-1} & A_{m-1} & \dots & O & O & \dots & O & O\\
O & B_{m-2} & \cdot & O & O & \dots & O & O \\
\vdots & \vdots & \ddots & \ddots & \vdots & \vdots & \vdots & \vdots \\
O & O & \dots & A_{k+1} & X & \dots & \dots & O \\
\vdots & \vdots & \vdots & \vdots & \ddots & \ddots & \vdots & \vdots \\
O & O & \dots & O & O & \cdot & A_{2} & X \\
O & O & \dots & P_{k+1} & O & \dots & O & O
\end{pmatrix}.$$
Finally, in the last step we multiply the first row block by $-P_{m-1}X^{-1}$ from the left and add it to the last row block. Hence, we obtain
$$\begin{pmatrix}
A_{m} & \vline& X & O & \dots & O &  O \\
B_{m-1} & \vline& A_{m-1} & X & \dots & O  & O \\
O & \vline&  B_{m-2} & A_{m-2} & \dots & O &  O \\
\vdots & \vline & \vdots & \vdots & \ddots & \vdots &  \vdots \\
O & \vline&  O & O & \dots & A_{2}  & X \\
\hline
P_{m} & \vline& O & O & \dots & O &  O
\end{pmatrix}=\begin{pmatrix}

 A_{m} &\vline &  &  &  &   \\
 B_{m-1} & \vline &  &  &  &   \\
 \vdots & \vline &  & \Delta &  &    \\
 O & \vline &  &  &  &   \\
 O & \vline&  &  &  &   \\
\hline
 P_{m}&  \vline & O & O & \dots & O  \\

\end{pmatrix}.$$
Further, the $X$ has full rank which implies that the matrix $\Delta$ has full rank. Therefore, we obtain the required result $rank(T)=(m-1)n+rank(P_{m})$.
\end{proof}

\begin{rem}{\label{rem2}}
Since the rank of the transpose of the matrix is equal to rank of itself, then we can prove similar result to Lemma~\ref{lem1} for matrices $T\in\mathbf{M}_{mn\times mn}(\mathbb{Z}_{p})$  of the following form:
\begin{align}\label{r2}
\begin{pmatrix}
A_{1} & B_{1} & O & \dots & O & O \\
X & A_{2} & B_{2} & \dots & O & O \\
O & X & A_{3} & \dots & O & O \\
\vdots & \vdots & \vdots & \ddots & \vdots & \vdots \\
O & O & O & \dots & A_{m-1} & B_{m-1} \\
O & O & O & \dots & X & A_{m}
\end{pmatrix},
\end{align}
where all submatrices are $n\times n$, $O$ is the zero matrix, $I$ is the identity matrix. If the submatrix $X$ has full rank, then
\[ rank(T)=(m-1)n+rank(P_{m}),\]
where $ P_{1}=A_{1}, \ \ P_{2}=-A_{1}X^{-1}B_{2}, \ \ P_{k}=-P_{k-2}X^{-1}B_{k-1}-P_{k-1}X^{-1}A_{k},  \ \ k\in\{3, \ldots, m\}.$
\end{rem}

Now, we give the algorithm of computing the rank of the rule matrix $T_{{\rm Rule}}$. Here, we can use the direct application of Lemma \ref{lem1} to the rule matrices.

\begin{thm}{\label{thmr}}
Let  $T^{\phi}_{{\rm R}}$ be the rule matrix in (\ref{T0}). If the matrix $B_1$ has the full rank, then
\[rank(T^{\phi}_{{\rm R}})=(m-1)n+rank(P_{m}),\]
where the submatrix $P_{m}$ is computed as in Lemma \ref{lem1}.
\end{thm}

\begin{proof}
It is clear that any rule matrix $T_{{\rm Rule}}^{\phi}$ is of the form (\ref{r1}). Therefore, we may have $X=B_1$. In the case $B_1$ has the full rank, the matrix $X$ has full rank. Hence, doing the same row block operations as in Lemma \ref{lem1} we find the submatrix $P_{m}$. Thus, we obtain
\[rank(T_{{\rm R}}^{\phi})=(m-1)n+rank(P_{m}).\]
\end{proof}

Now we identify the conditions when the matrix $B_1\in\mathbf{M}_{n\times n}(\mathbb{Z}_{p})$ has the full rank.
The matrix $B_1$ has the following form:

$$B_1=
\begin{pmatrix}
f & e & 0 & \dots & 0 & 0 \\
g & f & e & \dots & 0 & 0 \\
0 & g & f & \dots & 0 & 0 \\
\vdots & \vdots & \vdots & \ddots & \vdots & \vdots \\
0 & 0 & 0 & \dots & f & e \\
0 & 0 & 0 & \dots & g+e & f
\end{pmatrix}.$$

By computing the determinant via recurrence relation in Proposition \ref{detprop} we can try to simplify the form of the determinant of $B_1$ for some cases:
\begin{itemize}

\item if $e=0$ then $\det(B_1)=f^n$.

\item if $e\neq0$ and $g=0$ then $\det(B_1)=f^{n-2}(f^2-e^2)$.

\item if $eg\neq0$, $f=e+g$ and $e\neq g$ then $\det(B_1)=\frac{e^{n}-g^{n}}{e-g}g$.

\item if $eg\neq0$, $f=e+g$ and $e=g$ then $\det(B_1)=ne^n$.
\end{itemize}

\begin{rem} We can formulate the theorem \ref{thmr} for the rule matrix $T_{\mathrm{R}}^{\phi_{180^\circ}}$ in (\ref{T180}) by help of Remark \ref{rem2}.

\end{rem}

For computing the rank of the rule matrix $T_{\mathrm{R}}^{\phi_{90^\circ}}$ in (\ref{T90}) we give the following theorem without proof. It proves like Theorem \ref{thmr}.

\begin{thm}
    Let  $T^{\phi_{90^{\circ}}}_{{\rm R}}$ be the rule matrix in (\ref{T90}). If the matrix $C_2$ and $D_2$ have the full rank, then
\[rank(T^{\phi}_{{\rm R}})=(m-1)n+rank(P_{m}),\]
where the submatrix $P_{m}$ is computed as in Remark \ref{rem2}.
\end{thm}

\begin{exam}{\label{ex1}}
	In order to illustrate the previous theorem, we take $m=4$ and $n=3$ and consider the rule matrix $T^{\phi}_\mathrm{R}$ with over the ternary field $\mathbb{Z}_{3}$. Thanks to Theorem \ref{thm0} we have
$$T_{\mathrm{R}}^{\phi}=\left(\begin{array}{ccccc}
		A_1 & B_1 & O & O\\
		C_1 & A_1 & B_1 & O\\
		O & C_1 & A_1 & B_1\\
        O & O & D_1 & A_1
	\end{array}\right),$$
	where $O$ is the zero matrix and
$$A_1=\left(
	\begin{array}{ccccc}
		0 & d & 0 \\
		h & 0 & d \\
		0 & h+d & 0 \\
	\end{array}
	\right), \quad B_1=\left(
	\begin{array}{ccccc}
		f & e & 0 \\
		g & f & e \\
		0 & g+e & f \\
	\end{array}
	\right) ,$$
 $$C_1=\left(
	\begin{array}{ccccc}
		b & c & 0 \\
		a & b & c \\
		0 & a+c & b \\
	\end{array}
	\right), \quad D_1=\left(
	\begin{array}{ccccc}
		b+f & e+c+g & 0 \\
		g+a & b+f & c+e \\
		0 & g+a+e+c & b+f \\
	\end{array}
	\right),$$
with $a, b, c, d, e, f, g, h\in \mathbb{Z}_3$. 	

	Now, let $a=b=c=d=e=f=g=h=1$. Then,
$$A_1=\left(
	\begin{array}{ccccc}
		0 & 1 & 0 \\
		1 & 0 & 1 \\
		0 & 2 & 0 \\
	\end{array}
	\right), \quad B_1=C_1=\left(
	\begin{array}{ccccc}
		1 & 1 & 0 \\
		1 & 1 & 1 \\
		0 & 2 & 1 \\
	\end{array}
	\right) ,\quad D_1=\left(
	\begin{array}{ccccc}
		2 & 0 & 0 \\
		2 & 2 & 2 \\
		0 & 0 & 2 \\
	\end{array}
	\right).$$
Since $B_1$ has the full rank (i.e. $\det B_1=1$) we find the matrix $P_{4}$ as in the proof of Theorem \ref{thmr}. According to the recurrence relation in Lemma \ref{lem1}, we get the following matrices:
$$
		P_{1}=\left(
		\begin{array}{ccc}
			0 & 1 & 0 \\
			1 & 0 & 1 \\
			0 & 2 & 0 \\
		\end{array}
		\right), \ \ P_{2}=\left(
		\begin{array}{ccc}
			1 & 0 & 2 \\
			2 & 2 & 2 \\
			1 & 0 & 0 \\
		\end{array}
		\right),$$
		$$ P_{3}=\left(
		\begin{array}{ccc}
			2 & 0 & 2 \\
			0 & 0 & 0 \\
			1 & 0 & 0 \\
		\end{array}
		\right), \ \ P_{4}=\left(
		\begin{array}{ccc}
			2 & 0 & 1 \\
			1 & 1 & 1 \\
			0 & 2 & 1 \\
		\end{array}
		\right).$$
Therefore, by applying the algorithm for computing the rank of the rule matrix $T_{\mathrm{R}}^{\phi}$ (see Theorem \ref{thmr}) we have
\[rank(T_{\mathrm{R}}^{\phi})=(4-1)\cdot3+rank(P_{4})=9+2=11.\]
	
	Hence, the rule matrix $T_{\mathrm{R}}^{\phi}$ does not have full rank. This implies that it is not reversible. In other words the mapping $T_{\mathrm{R}}^{\phi}$ is not surjective, i.e. there exists a configuration of CA which is not an image of any configuration. Moreover, for this CA there exists Garden of Eden.
\end{exam}

In a cellular automaton, a Garden of Eden is a configuration that has no predecessor. It can be the initial configuration of the automaton but cannot arise in any other way. The successor of a configuration is another configuration, formed by applying the update rule simultaneously to every cell. The transition function of the automaton is the function that maps each configuration to its successor. If the successor of configuration $\mathbf{X}$ is configuration $\mathbf{Y}$, then $\mathbf{X}$ is a predecessor of $\mathbf{Y}$. A configuration may have zero, one, or more predecessors, but it always has exactly one successor. A Garden of Eden is defined to be a configuration with zero predecessors.

As it was mentioned in the introduction, computing the GOE configurations of a cellular automaton is an important notion in cellular automata theory.
  Moore introduced this notion in\cite{Moore1962}. Recently, Ying et al.\cite{Ying2009} have computed the number of GOE configurations of a 2D cellular automaton over the binary field $\mathbb{Z}_{2}$.


\section{Conclusions}

We investigate 2D finite linear Moore CA with some mixing boundary conditions over the $p$-ary field $\mathbb{Z}_p$ (i.e. $p$-state spin values). We construct the transition rule matrix corresponding to the boundary condition for Moore CA. Characterization problems for 2D finite linear Moore CA are analyzed for $p$-ary spin cases. As we noted  above,  it is impossible to simulate a truly infinite lattice on a computer.  Usually, for 2D CA, only one of the boundary conditions was considered: null, periodic, adiabatic, and reflexive. In \cite{Jumaniyozov2023} new type boundary condition was introduced as the mixed boundary condition. There, mixed boundary condition contains all four boundary conditions. In this paper, we generalize the mixed boundary condition. Suppose the boundary condition contains just two boundary conditions. For instance, one of the adjacent sides of the lattice has a null boundary condition, another adjacent sides has reflexive. There may be other cases. It may consider boundary conditions of opposite sides of the lattice are the same. Or three different types boundary conditions use on the sides of the lattice. But, in the paper we discussed mixed boundary conditions as an example above.
Studying the rest types of boundary conditions have been planned by us in the next papers.

A known fact is that determining the reversibility of 2D CA is a very difficult problem considering the general reversibility case. Reversibility is one of the main characterization of CA. If CAs are irreversible, then there is such a configuration that their pre-image does not exist. This configuration is called a Garden of Eden. We mentioned such kind of CA in Example \ref{ex1}. Given special $m,\ n$ values, and some number of rows and columns of the transition matrix, we develop an algorithm for computing the rank of the rule matrix with new types of boundary conditions for Moore neighborhood. Therefore, some special 2D CA is reversible and some of them are irreversible as given in Example \ref{ex1}. We present how to obtain the reversibility of 2D CA over the field $\mathbb{Z}_p$ elements. Then the reversibility problem of 2D Moore CA with a given mixed boundary condition is completely solved.

\section*{Acknowledgments}This work has been supported by  Ministry of Innovative Development of Uzbekistan, grant FZ-20200930492.

\end{document}